\documentclass[12pt]{amsart}
\usepackage{txfonts}      %{article} was 12pt latex e
\usepackage{amssymb}
\usepackage{eucal}
\usepackage{amsmath}
\usepackage{amscd}
\usepackage{xcolor}
\usepackage{multicol}
\usepackage[all]{xy}           %xypic macro for latex2.09
\usepackage{graphicx}
\usepackage{color}
\usepackage{colordvi}
\usepackage{xspace}
\usepackage{tikz}
\usepackage{makecell}
\usepackage{appendix}
\usepackage{enumerate,enumitem}
\usepackage{amsthm}
\usepackage[thicklines]{cancel}%划线
\usepackage[compress]{cite}
\usepackage{ifpdf}
\ifpdf
\usepackage[colorlinks,final,backref=page,hyperindex]{hyperref}
\else
\usepackage[colorlinks,final,backref=page,hyperindex,hypertex]{hyperref}
\fi

\usepackage[active]{srcltx} %SRC Specials for DVI Searching
\newtheorem{thm}{Theorem}[section]
\newtheorem{lem}[thm]{Lemma}
\newtheorem{cor}[thm]{Corollary}
\newtheorem{pro}[thm]{Proposition}
\theoremstyle{definition}
\newtheorem{defi}[thm]{Definition}
\newtheorem{ex}[thm]{Example}
\newtheorem{rmk}[thm]{Remark}

\title[extending structures,anti-dendriform algebras] {Extending structures for anti-dendriform algebras and anti-dendriform bialgebras}

\author{Qinxiu Sun}
\address{Department of Mathematics, Zhejiang University of Science and Technology, Hangzhou, 310023} \email{qxsun@126.com}
\author{Xingyu Zeng}
\address{Department of Mathematics, Zhejiang University of Science and Technology, Hangzhou, 310023} \email{948861157@qq.com}


\begin{document}
\begin{abstract}
	In this paper, we first explore the extending structures problem by the unified product for anti-dendriform algebras. 
In particular, the crossed product
and non-abelian extension are studied. Furthermore, we explore the inducibility problem of pairs of automorphisms associated
 with a non-abelian extension of anti-dendriform algebras,
and derive the fundamental sequences of Wells. Then we introduce the bicrossed products
and matched pairs of anti-dendriform algebras to solve the factorization problem. Finally, 
we introduce the notion of anti-dendriform D-bialgebras as the bialgebra structures corresponding 
   to double construction of associative algebras with respect to the commutative Cone cocycles. Both of them
are interpreted in terms of certain matched pairs of associative algebras as well 
as the compatible anti-dendriform algebras. The study of coboundary cases leads to the introduction of the AD-YBE, whose
skew-symmetric solutions give coboundary anti-dendriform D-bialgebras. 
The notion of $\mathcal O$-operators of anti-dendriform algebras is introduced to construct skew-symmetric solutions
of the AD-YBE. We also characterize the relationship between the skew-symmetric solutions of AD-YBE 
and $\mathcal O$-operators.
   
\end{abstract}

\maketitle

\tableofcontents

\allowdisplaybreaks

\section{Introduction}

The notion of dendriform algebras first appeared in Loday$'$s study of periodicity of algebraic $K$-theory and operads \cite{27},
 which is a vector space $A$ together with 
two binary operations $\succ$ and $\prec$, satisfying the following compatible conditions:
\begin{equation*}(x\prec y)\prec z=x\prec (y\prec z)+x\prec (y\succ z)
,\end{equation*} 
\begin{equation*}x\succ( y\prec z)=(x\succ y)\prec z
,\end{equation*}
 \begin{equation*}x\succ( y\succ z)=(x\prec y)\succ z+(x\succ y)\succ z,\end{equation*} 
for all $x,y,z\in A$.
The sum $\cdot=\succ+\prec$ gives an associative algebra $(A,\cdot)$ expresses a kind of splitting the associativity. 
More importantly, dendriform algebras are closely related to pre-Lie algebras and 
 Rota-Baxter algebras.
 The notion of anti-pre-Lie algebras \cite{25} was introduced by G. Liu and C. Bai  as the underlying
algebraic structures for nondegenerate commutative 2-cocycles on Lie algebras. Subsequently,
they developed a bialgebra theory for anti-pre-Lie algebras \cite{26}. Inspired by the study of anti-pre-Lie algebras,
Gao, Liu and Bai in \cite{15} introduced the notion of anti-dendriform algebras. Anti-dendriform algebras still
have the property of splitting the associativity, but it is the negative left and right multiplication operators that compose the
bimodules of the sum associative algebras, instead of the left and right multiplication operators
doing so for dendriform algebras. Moreover, there is close relationship between anti-dendriform algebras
and anti-pre-Lie algebras. The study of anti-dendriform algebras and their relationship to Novikov-type algebras provides a
foundational framework relevant to anti-pre-Lie algebras, as seen in \cite{15}. Explicitly, the following commutative diagram holds:
\begin{displaymath}\xymatrix{
 anti-dendriform ~~algebras \ar[d] \ar[r] & anti-pre-Lie~~algebras  \ar[d] \\
associative ~~algebras \ar[r] & Lie~~ algebras.}\end{displaymath}

Algebraic extension is an important method to study algebra structure.
The extending structures problem (ES-problem) arose in the study of group theory developed by 
Agore and Militaru \cite{2}, which unified the two well-known problems in group theory
–H$\rm{\ddot{o}}$lder's extension problem \cite{23} and the factorization problem of Ore \cite{31}. 
Since then, many researchers investigated the problem for various kinds of algebras,
such as Lie algebras (bialgebras), 
Lie (associative) conformal algebras, left symmetric algebras (bialgebras), Leibniz algebras,
Poisson algebras (bialgebras), perm-algebras, Jordan algebras and dendriform algebras,
see \cite{3,4,5,6,8,16,20,21,22,24,33,34,35,36} and references therein. 
We recall the ES-problem under the case of Lie algebras is as follows:
 Assume that $A$ is a Lie algebra, $E$ is a vector space and $i:A\longrightarrow E$ is a linear injective map.
 Characterize and classify all Lie algebraic structures on $E$ such that
 $i:A\longrightarrow E$ is a morphism of Lie algebras. It becomes tempting and natural to
look at the categorical dual of it by just reversing the arrow from the input data of the
extending structures problem. In \cite{7}, Militaru first considered the global 
extension problem (GE-problem) for Leibniz algebras, which is the categorical dual
of the ES-problem. Subsequently, the GE-problem for non-commutative Poisson algebras was studied in \cite{28}.

The classical extension problem is the 'local' version of the GE-problem.
The non-abelian extension is a relatively
general one among various extensions (e.g. central extensions,
abelian extensions, non-abelian extensions etc.). The non-abelian extensions
 were first developed by Eilenberg and Maclane \cite{14},
which induced to the low dimensional non-abelian cohomology group.
Then numerous works have been devoted to non-abelian extensions of
various kinds of algebras, such as Lie (super)algebras, perm-algebras,
 Rota-Baxter groups,
Rota-Baxter Lie algebras and Rota-Baxter Leibniz algebras, 
see \cite{ 11,18,19,24,29} and their references. 
Another interesting study related to extensions of algebraic structures is given by the extensibility or inducibility of a pair
of automorphisms. When a pair of automorphisms is inducible? This problem was first
 considered by Wells \cite{32} for abstract groups.
Since then, several authors have studied this subject further.
As byproducts, the Wells short exact sequences were obtained for
various kinds of algebras \cite{11,13,17,18,19,24}, which connected the relative automorphism groups and
the non-abelian second cohomology groups.

The extending structures and non-abelian extensions for anti-dendriform algebras are still absent. This is 
the first motivation for writing this paper. We are devoted to investigating the ES-problem, 
 the GE-problem and the non-abelian extensions for anti-dendriform algebras.
  More precisely, let $(A,\succ_{A},\prec_{A})$ be an anti-dendriform algebra, $E$ a vector space such that $i:A\longrightarrow E$
  is an injective map and $V$ is a complement of $A$ in $E$. Describe and classify the set of 
 all anti-dendriform algebraic structures on $E$ such that
     $i:A\longrightarrow E$ is a homomorphism of anti-dendriform algebras. 
     Denote by $\mathcal{G}(E,A)$ the small category whose objects are all anti-dendriform algebra structures 
   $(\succeq,\preceq)$ on $E$ such that $A$ is an anti-dendriform subalgebra of $E$. A 
   morphism $\varphi:(E,\succeq,\preceq)\longrightarrow (E,\succeq',\preceq')$ in the category $\mathcal{G}(E,A)$
    is an anti-dendriform algebra homomorphism $\varphi:(E,\succeq,\preceq)\longrightarrow (E,\succeq',\preceq')$ which stabilizes
    $A$, i.e., the left square of the following diagram is commutative:
    \begin{align}\label{Ee} \xymatrix{
  0 \ar[r] & A\ar@{=}[d] \ar[r]^{i} & E\ar[d]_{\varphi} \ar[r]^{\pi} & V \ar@{=}[d] \ar[r] & 0\\
 0 \ar[r] & A\ar[r]^{i} & E \ar[r]^{\pi} & V  \ar[r] & 0.}\end{align}
  In this case we say that the anti-dendriform algebra structures $(\succeq,\preceq)$ and $ (\succeq',\preceq')$
 on $E$ are equivalent, which is denoted by $(E,\succeq,\preceq)\equiv (E,\succeq',\preceq')$. 
  Let $\mathcal{G}'(E,A)$ be the subcategory of $\mathcal{G}(E,A)$. A 
   morphism $\varphi:(E,\succeq,\preceq)\longrightarrow (E,\succeq',\preceq')$ in the category $\mathcal{G}'(E,A)$
    is an anti-dendriform algebra homomorphism $\varphi:(E,\succeq,\preceq)\longrightarrow (E,\succeq',\preceq')$ which stabilizes
    $A$ and co-stabilizes $V$, i.e., the diagram (\ref{Ee}) is commutative. 
    In this case we say that the anti-dendriform algebra structures $(\succeq,\preceq)$ and $ (\succeq',\preceq')$
 on $E$ are cohomologous, which is denoted by $(E,\succeq,\preceq)\approx (E,\succeq',\preceq')$. 
 It is easy to see that $\equiv$ and $\approx$ are equivalence relations on $\mathcal{G}(E,A)$ and $\mathcal{G}'(E,A)$
 respectively. Denote $EXT(E,A)=\mathcal{G}(E,A)/\equiv$ and $EXT'(E,A)=\mathcal{G}'(E,A)/\approx$.
As in the classical extension problem, by classification of anti-dendriform algebra
 structures on $E$ we mean the classification up to an isomorphism of anti-dendriform algebras
 $(E,\succeq,\preceq)\longrightarrow(E,\succeq',\preceq')$ that stabilizes $ A$ and co-stabilizes $V$.
     
 Consider the categorical dual of the ES-problem: the global extension problem. In detail, assume that
  $A$ is an anti-dendriform algebra, $E$ a vector space and $\pi:E\longrightarrow A$ is an epimorphism
    of vector spaces with $V=Ker(\pi)$. Describe and classify the set of all anti-dendriform algebra structures on $E$ such that
     $\pi:E\longrightarrow A$ is a homomorphism of anti-dendriform algebras. If $E$ is an anti-dendriform algebra such that 
     $\pi:E\longrightarrow A$ is a homomorphism of anti-dendriform algebras. Then we have an extension of 
     the anti-dendriform algebra $A$ by $V=Ker(\pi)$ as follows:
     \begin{align}\label{E2}
	0\longrightarrow Ker(\pi)\stackrel{i}{\longrightarrow} E \stackrel{\pi}{\longrightarrow}A\longrightarrow0.
	\end{align}
But the anti-dendriform algebra structure on $ Ker(\pi)$ is not fixed, which depends essentially on the
anti-dendriform algebra structures $(\succeq,\preceq)$ on $E$ which we are looking for. In the
classical extension, which the anti-dendriform algebra structure on $ Ker(\pi)$ is fixed. 
 Denote by $\mathcal{C}_{\pi}(E,A)$ the small category whose objects are all anti-dendriform algebra structures 
   $(\succeq,\preceq)$ on $E$ such that $\pi:E\longrightarrow A$ is a homomorphism of anti-dendriform algebras. A 
   morphism $\varphi:(E,\succeq,\preceq)\longrightarrow (E,\succeq',\preceq')$ in the category $\mathcal{C}_{\pi}(E,A)$
    is an anti-dendriform algebra map $\varphi:(E,\succeq,\preceq)\longrightarrow (E,\succeq',\preceq')$ which stabilizes
    $V$ and co-stabilizes $A$, i.e., the following commutative diagram holds:
    \begin{equation}\label{E} \xymatrix{
  0 \ar[r] & V\ar@{=}[d] \ar[r]^{i} & E\ar[d]_{\varphi} \ar[r]^{\pi} & A \ar@{=}[d] \ar[r] & 0\\
 0 \ar[r] & V \ar[r]^{i} & E \ar[r]^{\pi} & A  \ar[r] & 0.}\end{equation}
As in the classical extension problem, by classification of anti-dendriform algebra
 structures on $E$ we mean the classification up to an isomorphism of anti-dendriform algebras
 $(E,\succeq,\preceq)\longrightarrow(E,\succeq',\preceq')$ that stabilizes $ V=Ker(\pi)$ and co-stabilizes $A$.
  In this case we say that the anti-dendriform algebra structures $(\succeq,\preceq)$ and $ (\succeq',\preceq')$
 on $E$ are cohomologous, which is denoted by $(E,\succeq,\preceq)\approx (E,\succeq',\preceq')$.     
     
A bialgebraic structure on a given algebraic structure is obtained by a 
comultiplication together with compatibility conditions between the
multiplications and comultiplications. 
Lie bialgebras were introduced in the early 1980s by Drinfeld \cite{12}, which
have also been applied in the theory of classical integrable systems and are closely
related to the classical Yang-Baxter equation. In \cite{37,38}, V. Zhelyabin developed Drinfeld’s ideas to
introduce the notion of associative D-bialgebras, which are associative analogue of Lie bialgebras. 
Antisymmetric infinitesimal bialgebras were developed by Aguiar \cite{1}, as the associative analogue
for Lie bialgebras, can be characterized as double constructions of Frobenius algebras \cite{10}. 
Bai and Ni combined the 
Lie bialgebras and (commutative and cocommutative)
infinitesimal bialgebras in a uniform way, which leaded to the notion of a Poisson bialgebra \cite{30}.
A Poisson bialgebra
is characterized as a Manin triple of Poisson algebras which is both a Manin triple of Lie algebras
(with respect to the invariant bilinear form) and a double construction of commutative Frobenius
algebras simultaneously. The method for Poisson bialgebras characterized by Manin triples with respect to the
invariant bilinear forms on both the commutative associative algebras and the Lie algebras is not
available for giving a bialgebra theory for transposed Poisson algebras.
Recently, Bai and Liu investigated the Manin triples 
with respect to the commutative 2-cocycles on Lie algebras, which leaded to the bialgebra theories for 
anti-pre-Lie algebras, transposed Poisson algebras and anti-pre-Lie Poisson algebras \cite{26}.

It is natural to explore the bialgebra structures for anti-dendriform algebras. This is another motivation for writing this paper. Explicitly,
 we develop a notion of an
anti-dendriform D-bialgebra, which is characterized by double construction of associative algebras with respect to the commutative Cone cocycles
and certain matched pairs of anti-dendriform algebras.
The study of coboundary anti-dendriform D-bialgebra leads to the introduction of anti-dendriform Yang-Baxter
equation(AD-YBE). A skew-symmetric solution of the anti-dendriform Yang-Baxter equation naturally gives 
(coboundary) anti-dendriform D-bialgebras.

 The paper is organized as follows. In Section 2, we study the representation of anti-dendriform algebras.  
 In Section 3, we introduce the notion of an extending datum of anti-dendriform algebras and
 study the ES-problem for anti-dendriform algebras. Explicitly,
   using unified products, we show that any anti-dendriform structure on $E$ 
such that $A$ is an anti-dendriform subalgebra is isomorphic to a unified product of $A$ 
by $V$, where $V$ is a complement of $A$ in $E$. With this tool, we construct objects $H^2_{p}(V, A)$ and
 $H^2(V, A)$ to give a theoretical answer for the ES-problem of anti-dendriform algebras.
 In Section 4, we consider a special case of unified products: crossed products $A\sharp V$. 
 Let $\pi:E\longrightarrow A$ be s surjective map with $V=Ker(\pi)$.
 We show that any anti-dendriform algebra
 structure $(\succeq,\preceq)$ on $E$ such that $\pi:E\longrightarrow A$ is a homomorphism of 
 anti-dendriform algebras is cohomologous to a crossed product. Subsequently, the answer to the GE-problem
 is provided in Theorem \ref{GE}. 
 The (non-abelian) extension is studied. We show that any extension of $A$ by $V$ is cohomogous to the following crossed 
 extension:
 \begin{align*}
	0\longrightarrow V\stackrel{i_V}{\longrightarrow} A\sharp V \stackrel{\pi_A}{\longrightarrow}A\longrightarrow0.
	\end{align*}
 We also study the extensibility problem of a pair of automorphisms about a non-abelian extension of anti-dendriform 
 algebras and derive Wells short exact sequences in the context of non-abelian
extensions of anti-dendriform algebras.
 In Section 5, we introduce matched pairs and bicrossed products of anti-dendriform algebras to solve
the factorization problem. 
  In Section 6, we introduce the notion of anti-dendriform D-bialgebras as the bialgebra structures corresponding 
   to double construction of associative algebras with respect to the commutative Cone cocycles. Both of them
are interpreted in terms of certain matched pairs of associative algebras as well 
as the compatible anti-dendriform algebras. The study of coboundary cases leads to the introduction of the AD-YBE, whose
skew-symmetric solutions give coboundary anti-dendriform D-bialgebras. 
The notion of $\mathcal O$-operators of
anti-dendriform algebras is introduced to construct skew-symmetric solutions
of the AD-YBE. We also characterize the relationship between the skew-symmetric solutions of AD-YBE and $\mathcal O$-operators.
   
Throughout the paper, $k$ is a field.  All vector spaces and algebras are over $k$. 
 All algebras are finite-dimensional, although many results still hold in the infinite-dimensional case.
 
\section{Representation of anti-dendriform algebras}

We study the representations of anti-dendriform algebras.

An {\bf anti-dendriform algebra} is a vector space $A$ together with two bilinear maps $\succ,\prec:A\times A\longrightarrow A$ such that
 \begin{equation}\label{A1}x\succ (y\succ z)=-(x\cdot y)\succ z=-x\prec (y\cdot z)=(x\prec y)\prec z,\end{equation} 
 \begin{equation}\label{A2}(x\succ y)\prec z=x\succ (y\prec z),\end{equation} 
  for all $x,y,z\in A$, where $x\cdot x\succ y+x\prec y$. $(A,\cdot)$ is an associative algebra, which is called the associated associative 
  algebra of the anti-dendriform algebra $(A,\succ,\prec)$. Furthermore, $(A,\succ,\prec)$ is called a compatible 
  anti-dendriform algebra structure on $(A,\cdot)$.
  
  When $x\succ y=y\prec x=x\ast y$, $(A,\ast)$ is called an anti-Zinbiel algebra.

 \begin{defi} A {\bf representation (bimodule)} of an anti-dendriform algebra $(A,\succ,\prec)$ is a quintuple
  $(V,l_{\succ}, r_{\succ},l_{\prec},r_{\prec})$, 
 where $V$ is a
vector space and $l_{\succ}, r_{\succ},l_{\prec},r_{\prec} : A \longrightarrow \hbox{End}
(V)$ are four linear maps satisfying the following relations, for all
$x,y\in A$,
\begin{equation}\label{R1}l_{\succ}(x)l_{\succ}(y)=-l_{\succ}(x\cdot y)=-l_{\prec}(x)l_{\cdot}(y)=l_{\prec}(x\prec y),\end{equation}
\begin{equation}\label{R2}r_{\succ}(x\succ y)=-r_{\succ}(y)r_{\cdot}(x)=-r_{\prec}(x\cdot y)=r_{\prec}(y)r_{\prec}(x),\end{equation}
\begin{equation}\label{R3}l_{\succ}(x)r_{\succ}(y)=-r_{\succ}(y)l_{\cdot}(x)=-l_{\prec}(x)r_{\cdot}(y)=r_{\prec}(y)l_{\prec} (x),\end{equation}
\begin{equation}\label{R4}l_{\prec}(x \succ y)=l_{\succ}(x)l_{\prec}(y),\end{equation}
\begin{equation}\label{R5}r_{\prec}(y)r_{\succ}(x)=r_{\succ}(x\prec y),\end{equation}
\begin{equation}\label{R6}r_{\prec}(y)l_{\succ}(x)=l_{\succ}(x)r_{\prec} (y),\end{equation}
where $l_{\cdot}(x)=l_{\succ}(x)+l_{\prec}(x),~r_{\cdot}(x)=r_{\succ}(x)+r_{\prec}(x)$ and $x\cdot y=x\succ y+x\prec y$.
\end{defi}
By Eqs.\eqref{R3} and \eqref{R6}, we get
 \begin{equation}\label{R7}r_{\cdot}(y)l_{\cdot}(x)=l_{\cdot}(x)r_{\cdot} (y).\end{equation}
 
\begin{pro}  Let $(A,\succ,\prec)$ be an anti-dendriform algebra and $V$ a vector space. Assume that
$l_{\succ}, r_{\succ},l_{\prec},r_{\prec} : A \longrightarrow \hbox{End}
(V)$ are four linear maps. Then 
$(V,l_{\succ}, r_{\succ},l_{\prec},r_{\prec})$ is a representation of $(A,\succ,\prec)$ if and only if 
$(A\oplus V,\succeq,\preceq)$ is an anti-dendriform algebra with $(\succeq,\preceq)$ given by
\begin{align*}
    	&(x,a)\succeq (y,b)=(x\succ y,l_{\succ}(a)y+r_{\succ}(b)x),\\
    &(x,a)\preceq(y,b)=(x\prec y,l_{\prec}(a)y+r_{\prec}(b)x),
    \end{align*}
for all $x,y\in A$ and $a,b\in V$. In this case, the anti-dendriform algebra $(A\oplus V,\succeq,\preceq)$ is called
the semi-direct product. Denote it simply by $A\ltimes V$.
\end{pro}

\begin{proof} It is a special case of Proposition \ref{U1}.

\end{proof}
 
Let $A$ and $V$ be vector spaces. For
a linear map $f: A \longrightarrow \hbox{End} (V)$, define a linear
map $f^{*}: A \longrightarrow \hbox{End} (V^{*})$ by $\langle
f^{*}(x)u^{*},v\rangle=\langle u^{*},f(x)v\rangle$ for all $x\in A,
u^{*}\in V^{*}, v\in V$, where $\langle \ , \ \rangle$ is the usual
pairing between $V$ and $V^{*}$.

\begin{pro} \label{zr} Let $(V,l_{\succ}, r_{\succ},l_{\prec},r_{\prec})$ be a
representation of an anti-dendriform algebra $(A,\succ,\prec)$. Then
\begin{enumerate}
	\item $(V,-l_{\succ},-r_{\prec})$ is a representation of the associated
 associative algebra $(A,\cdot)$.
 \item $(V,l_{\succ}+l_{\prec},r_{\succ}+r_{\prec})$ is a representation of the associated
 associative algebra $(A,\cdot)$.
	\item $(V^{*},-(r_{\prec}^{*}+r_{\succ}^{*}),l_{\prec}^{*},r_{\succ}^{*},-(l_{\prec}^{*}+l_{\succ}^{*}))$ is also a
	representation of $(A,\succ,\prec)$. We call it the {\bf dual representation}.
	\item $(V^{*},-r_{\prec}^{*},-l_{\succ}^{*})$ is a
	representation of $(A,\cdot)$.
\item $(V^{*},r_{\prec}^{*}+r_{\succ}^{*},l_{\prec}^{*}+l_{\succ}^{*})$ is a
	representation of $(A,\cdot)$. 
\end{enumerate}
\end{pro}

\begin{proof} 
(a) It can be obtained by Eqs. \eqref{R1}-\eqref{R2} and \eqref{R6}.

(b) It follows by Eqs. \eqref{R1}-\eqref{R6}.

(c) By Eqs. \eqref{R1}-\eqref{R2} and \eqref{R4}-\eqref{R5}, for all $x,y\in A,v^{*}\in V^{*}$ and $w\in V$, we have
\begin{align*}&
\langle (r_{\prec}^{*}+r_{\succ}^{*})(x)(r_{\prec}^{*}+r_{\succ}^{*})(y)v^{*},w\rangle
\\=&\langle (v^{*},r_{\prec}(y)r_{\prec}(x)w+r_{\prec}(y)r_{\succ}(x)w+r_{\succ}(y)
r_{\prec}(x)w+r_{\succ}(y)r_{\succ}(x)w\rangle
\\=&\langle (v^{*},r_{\succ}(x\prec y)w\rangle
\\=&\langle (v^{*},r_{\succ}(x\cdot y)w-r_{\succ}(x\succ y)w\rangle
\\=&\langle (v^{*},r_{\succ}(x\cdot y)w+r_{\prec}(x\cdot y)w\rangle
\\=&\langle ((r_{\succ}^{*}+r_{\prec}^{*})((x\cdot y))v^{*},w\rangle,\end{align*}
which indicates that
\begin{equation*}(r_{\prec}^{*}+r_{\succ}^{*})(x)(r_{\prec}^{*}+r_{\succ}^{*})(y)=(r_{\succ}^{*}+r_{\prec}^{*})((x\cdot y),\end{equation*}
that is, Eq.\eqref{R1} holds for $l_{\succ}=-(r_{\prec}^{*}+r_{\succ}^{*})$.
Analogously, Eqs.\eqref{R2}-\eqref{R6} hold for 
$(-(r_{\prec}^{*}+r_{\succ}^{*}),l_{\prec}^{*},r_{\succ}^{*},-(l_{\prec}^{*}+l_{\succ}^{*}))$.

(d) (e) These are obtained directly from (a), (b) and (c).
\end{proof}

\begin{ex} Let $(A,\succ,\prec)$ be an anti-dendriform algebra,
	and $L_{\succ},R_{\succ},L_{\prec},R_{\prec} : A\longrightarrow \hbox{End} (A)$ be the
linear maps defined by $L_{\succ}(x)(y)=R_{\succ}(y)(x)=x\succ y,~~L_{\prec}(x)(y)=R_{\prec}(y)(x)=x\prec y$ 
for all $x,y\in A$. Then 

\begin{enumerate}
\item $ (A,L_{\succ},R_{\succ},L_{\prec},R_{\prec})$ is a representation of $(A,\succ,\prec)$,
which is called the {\bf regular representation} of $(A,\succ,\prec)$. 
Moreover, $(A^{*},-(R_{\prec}^{*}+R_{\succ}^{*}),L_{\prec}^{*},R_{\succ}^{*},-(L_{\prec}^{*}+L_{\succ}^{*})$ 
is the dual representation of $(A,L_{\succ},R_{\succ},L_{\prec},R_{\prec})$. 
\item $ (A,-L_{\succ},-R_{\prec})$ is a representation of the associated
 associative algebra $(A,\cdot)$. $(A^{*},-R_{\prec}^{*},-L_{\succ}^{*})$ 
is the dual representation of $(A,\cdot)$. 
\end{enumerate}
\end{ex}

%%%%%%%%%%%%%%%%%%%%%%%%%%%%%%%%%%%%%%%%%%%%%%%%%%%%%%%%%%%%%%%%%%%%%%%%%%%%%%%%%%%%%%%%%%%%%%%%%%%%%%%

\section{Extending structures and the ES-problem for anti-dendriform algebras}
In this section, we introduce extending datums and unified products of anti-dendriform algebras. 
Moreover, we give the answer to the ES-problem for anti-dendriform algebras.

\begin{defi}
  		Let $(A,\succ,\prec)$ be an anti-dendriform algebra and $V$ a vector space. 
  An \textbf{extending datum } of $(A,\succ,\prec)$ through $V$
    is a system $\Omega(A,V)=(l_{\succ},r_{\succ},l_{\prec},r_{\prec},\rho_{\succ},\mu_{\succ},\rho_{\prec},\mu_{\prec},
    \varpi_1,\varpi_2,\succ_V,\prec_V)$ 
    consisting of eight linear maps 
  		\begin{align*}
  			l_{\succ},r_{\succ},l_{\prec},r_{\prec}: A \longrightarrow End(V), 
 \ \ \ \rho_{\succ},\mu_{\succ},\rho_{\prec},\mu_{\prec}: V \longrightarrow End(A)
  			\end{align*}
  and four bilinear maps 
  \begin{align*} \varpi_1,\varpi_2: V\times V \longrightarrow A, \ \ \ \succ_V,\prec_V: V \times V \longrightarrow V.
  		\end{align*}
  		Denote by $A\natural V $ the vector space $ A\oplus V$ 
  with the bilinear maps $\succeq,~\preceq: (A\oplus V)\times(A \oplus V) \longrightarrow A \oplus V$ 
  given as follows:
  \begin{align}
    	&(x,a)\succeq (y,b)=(x\succ y+\rho_{\succ}(a)y+\mu_{\succ}(b)x+\varpi_1(a,b),l_{\succ}(a)y+r_{\succ}(b)x+a\succ_Vb),\label{SE1}\\
    &(x,a)\preceq(y,b)=(x\prec y+\rho_{\prec}(a)y+\mu_{\prec}(b)x+\varpi_2(a,b),l_{\prec}(a)y+r_{\prec}(b)x+a\prec_Vb),\label{SE2}
    \end{align}
  for all $x,y\in A$ and $a,b\in V$.
 The object $A\natural V $ is called the \textit{unified product} of $(A,\succ,\prec)$ and $V$ if 
  $(A\oplus V,\succeq,\preceq)$ is an anti-dendriform algebra. 
  In this case, the extending datum $\Omega(A, V) $ is called an extending structure of
   $(A,\succ,\prec)$ through $V$.
  Denote by $\mathcal{A}(A, V)$ the set of all extending structures of $(A,\succ,\prec)$ through $V$.
  	\end{defi}
  
  \begin{pro}\label{U1} Let $(A,\succ,\prec)$ be an anti-dendriform algebra, $V$ be a vector space and $\Omega(A,V)=(l_{\succ},r_{\succ},l_{\prec},r_{\prec},\rho_{\succ},\mu_{\succ},\rho_{\prec},\mu_{\prec},
    \varpi_1,\varpi_2,\succ_V,\prec_V)$ be an \textit{extending datum of
   $(A,\succ,\prec)$
  through $V$}. Then $A\natural V$ is a unified product of $(A,\succ,\prec)$ and $V$ if 
  and only if the following conditions hold for all $x,y\in A$ and $a,b,c\in V$,
  \begin{align} \label{S1} &\text{$(l_{\succ},r_{\succ},l_{\prec},r_{\prec})$ is a representation of $(A,\succ,\prec)$ on $V$,}\end{align}
 \begin{align} \label{S2} x\succ(\mu_{\succ}(a)y)+\mu_{\succ}(l_{\succ}(y)a)x=\mu_{\prec}(a)(x\prec y)=-\mu_{\succ}(a)(x\cdot y)
  =-x\prec(\mu(a)y)-\mu_{\prec}(l_{\cdot}(y)a)x, \end{align}
   \begin{align} \label{S3}&x\succ(\rho_{\succ}(a)y)+\mu_{\succ}(r_{\succ}(y)a)x=(\mu_{\prec}(a)x)\prec y+\rho_{\prec}(l_{\prec}(x)a)y\\
   \notag=&-(\mu(a)x)\succ y-\rho_{\prec}(l_{\cdot}(x)a)y=-x\prec(\rho(a)y)-\mu_{\prec}(r_{\cdot}(y)a)x,\end{align}
   \begin{align} \label{S4}&\rho_{\succ}(a)(x\succ y)=(\rho_{\prec}(a)x)\prec y+\rho_{\prec}(r_{\prec}(x)a)y\\
  \notag=&-\rho_{\prec}(x\cdot y)=-(\rho(a)x)\succ y-\rho_{\succ}(r(x)a)y,\end{align}
   \begin{align} \label{S5} &x\succ \varpi_1(a,b)+\mu_{\succ}(a\succ_Vb)x=\varpi_1(l_{\prec}(x)a,b)+\mu_{\prec}(b)\mu_{\prec}(a)x\\
  \notag=&-\varpi_1(l_{\cdot}(x)a,b)-\mu_{\succ}(b)\mu(a)x
  =-x\prec \varpi(a,b)-\mu_{\prec}(a\cdot_Vb) x,\end{align}
   \begin{align}\label{S6}&l_{\succ}(x) (a\succ_V b)=l_{\prec}(\mu_{\prec}(a)x)b+(l_{\prec}(x)a)\succ_Vb\\
   \notag=&-l_{\succ}(\mu(a)x)b-(l_{\cdot}(x)a)\succ_Vb=-l_{\prec}(x)(a\cdot_V b),\end{align}
   \begin{align} \label{S7}&\rho_{\succ}(a)\mu_{\succ}(b)x+\varpi_1(a,l_{\succ}(x)b)=\mu_{\prec}(b)\rho_{\prec}(a)x+\varpi_2(r_{\prec}(x)a,b)\\
\notag=&-\mu_{\succ}(b)\rho(a)x-\varpi_1(r_{\cdot}(x)a,b) =-\rho_{\prec}(a)\mu(b)x-\varpi_2(a,l_{\cdot}(x)b),\end{align}
   \begin{align}\label{S8}&r_{\succ}(\mu_{\succ}(b)x)a+a\succ_V(l_{\succ}(x)b)=(r_{\prec}(x)a)\prec_Vb+l_{\prec}(\rho_{\prec}(a)x)b
   \\\notag=&-(r_{\cdot}(x)a)\succ_Vb-l_{\succ}(\rho(a)x)b=-a\prec_V(l_{\cdot}(x)b)-r_{\prec}(\mu(b)x)a,\end{align}
    \begin{align}\label{S9}&\rho_{\succ}(a)\rho_{\succ}(b)x+\varpi_1(a,r_{\succ}(x)b)=\varpi_2(a,b)\prec x+\rho_{\prec}(a\prec_Vb)x
   \\\notag=&-\varpi(a,b)\succ x-\rho_{\succ}(a\cdot_Vb)x=-\varpi_2(a,r_{\cdot}(x)b)-\rho_{\prec}(a)\rho(b)x,\end{align}
   \begin{align}\label{S10}&r_{\succ}(\rho_{\succ}(b)x)a+a\succ_V(r_{\succ}(x)b)=r_{\prec}(x)(a\prec_V b)
   \\\notag=&-r_{\succ}(x)a\cdot_V b)=-r_{\prec}(\rho(b)x)a-a\prec_V(r_{\cdot}(x)b),\end{align}
   \begin{align}\label{S11} &\rho_{\succ}(a)\varpi_1(b,c)+\varpi_1(a,b \succ_Vc)=\mu_{\prec}(c)\varpi_2(a,b)+\varpi_2(a\prec_Vb,c)\\
  \notag=&-\rho_{\prec}(a)\varpi(b,c)-\varpi_2(a,b\cdot_Vc)=-\mu_{\succ}(c)\varpi(a,b)-\varpi_{1}(a\cdot_Vb,c),\end{align}
  \begin{align}\label{S12}&r_{\succ}(\varpi_1(b,c))a+a\succ_V(b\succ_Vc)=l_{\prec}(\varpi_2(a,b))c+(a\prec_Vb)\prec_Vc\\
  \notag=&-r_{\prec}(\varpi(b,c))a-a\prec_V(b\cdot _Vc)=-l_{\succ}(\varpi(a,b))c-(a\cdot_Vb)\succ_vc,\end{align}
\begin{align}\label{S13}\mu_{\prec}(a)(x\succ y)=x\succ(\mu_{\prec}(a)y)+\mu_{\prec}(l_{\prec}(y)a)x,\end{align}
\begin{align}\label{S14}(\mu_{\succ}(a)x)\prec y+\rho_{\prec}(l_{\succ}(x)a)y=x\succ(\rho_{\prec}(a)y)+\mu_{\succ}(r_{\prec}(y)a)x,\end{align}
\begin{align}\label{S15}(\rho_{\succ}(a)x)\prec y+\rho_{\prec}(r_{\succ}(x)a)y=\rho_{\succ}(a)(x\prec y),\end{align}
\begin{align}\label{S16}\mu_{\prec}(c)\varpi_1(a,b)+\varpi_2(a\succ_v b,c)=\rho_{\succ}(a)\varpi_2(b,c)+\varpi_1(a,b\prec_Vc),\end{align}
\begin{align}\label{S17} l_{\prec}(\varpi_1(a,b))c+(a\succ_Vb)\prec_Vc=r_{\succ}(\varpi_2(b,c))a+a\succ_V(b\prec_Vc),\end{align}
where $\cdot=\succ+\prec,~\cdot_V=\succ_V+\prec_V,~\varpi(a,b)=\varpi_1(a,b)+\varpi_2(a,b)$ 
and $ l_{\cdot}=l_{\succ}+l_{\prec},~r_{\cdot}=r_{\succ}+r_{\prec},~\rho_{\cdot}=\rho_{\succ}+\rho_{\prec},~\mu_{\cdot}=\mu_{\succ}+\mu_{\prec}$.
  \end{pro}
  
\begin{proof} The object $A\natural V$ is a unified product if and only if 
  the following conditions hold for all $x,y,z\in A$ and $a,b,c\in V$:
  \begin{align}\label{S18}&(x,a)\succeq( (y,b)\succeq (z,c))=-((x,a)\bullet (y,b))\succeq (z,c)
  \\=&-(x,a)\preceq ((y,b)\bullet (z,c))=((x,a)\preceq (y,b))\preceq (z,c),\end{align} 
 \begin{equation}\label{S19}((x,a)\succeq (y,b))\preceq (z,c)=(x,a)\succeq ((y,b)\preceq (z,c)),\end{equation} 
 where $\bullet=\succeq+\preceq$.
In fact, for the Eq. \eqref{S18}, we consider it by the following cases:

(i) \eqref{S18} holds for the triple $((x,0),(y,0),(0,a))$ if and only if \eqref{S2} and \eqref{R1} hold.

(ii) \eqref{S18} holds for the triple $((x,0),(0,a),(y,0))$ if and only if \eqref{S3} and \eqref{R3} hold.

(iii) \eqref{S18} holds for the triple $((0,a),(x,0),(y,0))$ if and only if \eqref{S4} and \eqref{R2} hold.

(iv) \eqref{S18} holds for the triple $((x,0),(0,a),(0,b))$ if and only if \eqref{S5} and \eqref{S6} hold.

(v) \eqref{S18} holds for the triple $((0,a),(x,0),(0,b))$ if and only if \eqref{S7} and \eqref{S8} hold.

(vi) \eqref{S18} holds for the triple $((0,a),(0,b),(x,0))$ if and only if \eqref{S9} and \eqref{S10} hold.

(vii) \eqref{S18} holds for the triple $((0,a),(0,b),(0,c))$ if and only if \eqref{S11} and \eqref{S12} hold.
  
For the Eq.\eqref{S19}, we consider it by the following cases: 

(i) \eqref{S19} holds for the triple $((x,0),(y,0),(0,a))$ if and only if \eqref{S13} and \eqref{R4} hold.

(ii) \eqref{S19} holds for the triple $((x,0),(0,a),(y,0))$ if and only if \eqref{S14} and \eqref{R5} hold.

(iii) \eqref{S19} holds for the triple $((0,a),(x,0),(y,0))$ if and only if \eqref{S15} and \eqref{R6} hold.

(iv) \eqref{S19} holds for the triple $((0,a),(0,b),(0,c))$ if and only if \eqref{S16} and \eqref{S17} hold.

 By direct calculation, we can verify that the above statements hold.
  \end{proof}
  
  Since $(A\natural V,\succeq,\preceq)$ is an anti-dendriform algebra, $(A\natural V,\bullet)$ is an associative algebra.
  Then
 \begin{rmk} \label{U3}
Let $\Omega(A, V)=(l_{\succ},r_{\succ},l_{\prec},r_{\prec},\rho_{\succ},\mu_{\succ},\rho_{\prec},\mu_{\prec},
   \varpi_1,\varpi_2,\succ_V,\prec_V) $ be an extending structure of
   $(A,\succ,\prec)$ through $V$. Then $\Omega(A, V)=(l_{\succ}+l_{\prec},r_{\succ}+r_{\prec},\rho_{\succ}+\rho_{\prec},\mu_{\succ}+\mu_{\prec},
    \varpi_1+\varpi_2,\succ_V+\prec_V) $ is an extending structure of $(A,\cdot)$ through $V$.
 \end{rmk}

\begin{thm}\label{U2} 
    Let $(A,\succ,\prec)$ be an anti-dendriform algebra, $E$ a vector space containing $A$ 
  as a subspace and $(E,\succeq,\preceq)$ an anti-dendriform algebra such that $A$ 
  is an anti-dendriform subalgebra of $E$. Then there is an \textit{extending structure} $\Omega(A,V)=(l_{\succ},r_{\succ},l_{\prec},r_{\prec},\rho_{\succ},\mu_{\succ},\rho_{\prec},\mu_{\prec},
    \varpi_1,\varpi_2,\succ_V,\prec_V)$ 
  of $A$ through a subspace $V$ of $E$ and an isomorphism of anti-dendriform algebras 
  $(E,\succeq,\preceq) \cong A\natural V$ that stabilizes $A$ and co-stabilizes $V$.
  \end{thm}
  
  \begin{proof} Clearly,
  there is a linear map $p : E \longrightarrow A$ such that $p(x)=x$ for all $x \in A$. 
  Then $V=\mathrm{ker}(p)$ is a subspace of $E$, which is a complement of $A$ in $E$. Define 
 the extending datum $\Omega(A,V)$ of $(A,\succ,\prec)$ through $V$ as follows:
 \begin{equation}\label{S20}
 	l_{\succ}(x)a=x\succeq a-p(x\succeq a),~r_{\succ}(x)a=a\succeq x-p(a\succeq x), \end{equation}
 \begin{equation}\label{S21}l_{\prec}(x)a=x\preceq a-p(x\preceq a),~r_{\prec}(x)a=a\preceq x-p(a\preceq x),~~\rho_{\succ}(a)x=p(a\succeq x), \end{equation}
 	 \begin{equation}\label{S22}\mu_{\succ}(a)x=p(x\succeq a),\ \ \rho_{\prec}(a)x=p(a\preceq x), 
 \ \ \mu_{\prec}(a)x=p(a\preceq x), \ \ \varpi_1(a,b)=p(a \succeq b) \end{equation}
 \begin{equation}\label{S23}\varpi_2(a,b)=p(a \preceq b),~~ a\succ_Vb=a\succeq b-p(a\succeq b),~~a\prec_Vb=a\preceq b-p(a\preceq b)\end{equation}
 for all $x,y\in A,~a,b\in V$. Next, we prove that
 the extending datum $\Omega(A, V) $ is an extending structure of $(A,\succ,\prec)$ through $V$.
 Define a linear isomorphism
 \begin{equation}\label{S24}\varphi:A\natural V\longrightarrow (E,\succeq,\preceq),~~~\varphi(x,a)=x+a,~ ~\forall~x\in A,~a\in V,\end{equation}
 whose inverse is  
  \begin{equation}\label{S25}\varphi^{-1}:(E,\succeq,\preceq)\longrightarrow A\natural V, ~~~~\varphi^{-1}(e)=(p(e),e-p(e)),~~\forall~e\in E.\end{equation}
  We shall verify that the anti-dendriform algebra structure $(\succeq,\preceq)$ that can be defined on $A\oplus V$ such that 
  $\varphi$ is an isomorphism of anti-dendriform algebras coincides with the one given by Eqs.\eqref{SE1}-\eqref{SE2}.
 Indeed, using Eqs.\eqref{S20}-\eqref{S23}, we get
  \begin{equation}\label{S26}
 	p(x\succeq y+x\succeq b+a\succeq y+a\succeq b)=x\succ y+\mu_{\succ}(b)x+\rho_{\succ}(a)y+\varpi_1(a,b),
 \end{equation}
  \begin{align}\label{S27}
 	&x\succeq y+x\succeq b+a\succeq y+a\succeq b-p(x\succeq y+x\succeq b+a\succeq y+a\succeq b)\\
\notag=&l_{\succ}(x)b+r_{\succ}(y)a+a\succ_Vb
 \end{align}
 for all $x,y\in A$ and $a,b\in V$.
Based on the Eqs. \eqref{S24}-\eqref{S27}, we have
 \begin{align*}
 	&(x,a)\succeq (y,b)=\varphi^{-1}(\varphi(x,a)\succeq \varphi(y,b))
 \\=&\varphi^{-1}((x+a)\succeq (y+b))=\varphi^{-1}(x\succeq y+x\succeq b+a\succeq y+a\succeq b)
 \\=&(x\succ y+\mu_{\succ}(b)x+\rho_{\succ}(a)y+\varpi_1(a,b),l_{\succ}(x)b+r_{\succ}(y)a+a\succ_Vb).
 \end{align*}
 Analogously, 
 \begin{equation*}(x,a)\preceq (y,b)=(x\prec y+\mu_{\prec}(b)x+\rho_{\prec}(a)y+\varpi_2(a,b),l_{\prec}(x)b+r_{\prec}(y)a+a\prec_Vb).\end{equation*}
In addition, it is easy to check that $\varphi$ stabilizes $A$ and co-stabilizes $V$.
The proof is finished.
  	 \end{proof}
  	 
 By Theorem \ref{U2}, any anti-dendriform algebra structure on a vector space $E$ containing $A$ 
 as an anti-dendriform subalgebra is isomorphic to a unified product of $A$ 
 through a given complement $V$ of $A$ in $E$. Hence the classification 
 of all anti-dendriform algebra structures on $E$ that contains $A$ as an 
 anti-dendriform subalgebra is equivalent to the classification of all 
 unified products $A\natural V$, associated to all extending structures $\Omega(A,V)$
  for a given complement $V$ of $A$ in $E$. 
  	
  In the sequel, we will classify all the unified products $A\natural V$.
  	 \begin{lem}\label{63}
  	 	Let $(A,\succ,\prec)$ be an anti-dendriform algebra, $V$ a vector space and $\Omega(A,V)=(l_{\succ},r_{\succ},l_{\prec},r_{\prec},\rho_{\succ},\mu_{\succ},\rho_{\prec},\mu_{\prec},
    \varpi_1,\varpi_2,\succ_V,\prec_V)$, 
    $\Omega(A,V)=(l_{\succ}',r_{\succ}',l_{\prec}',r_{\prec}',\rho_{\succ}',\mu_{\succ}',\rho_{\prec}',\mu_{\prec}',
    \varpi'_{1},\varpi'_{2},\succ'_V,\prec'_V)$
   be two extending structures of $(A,\succ,\prec)$ through $V$, 
   and $A\natural V,~A \natural' V$ are their associated unified products respectively.
   Denote by $\mathcal{H}$ the set of all homomorphisms of 
    anti-dendriform algebras $\psi: A\natural V\longrightarrow A \natural' V$ 
    which stabilizes $A$ in the following diagram:
     \begin{equation} \label {T1} \xymatrix{
  0 \ar[r] & A\ar@{=}[d] \ar[r]^{i} & A\natural V\ar[d]_{\psi} \ar[r]^{\pi_V} & V \ar@{=}[d] \ar[r] & 0\\
 0 \ar[r] & A \ar[r]^{i'} & A \natural' V \ar[r]^{\pi'_{V}} & V  \ar[r] & 0 .}\end{equation}
 and denote by $\mathcal{P}$ the set of 
    pairs $(\zeta,\eta)$, where $\zeta: V \longrightarrow A$,
     $\eta: V \longrightarrow V$ are linear maps and they satisfy the following conditions,
  	 	\begin{align}
  	 	&\eta(l_{\succ}(x)a)=l'_{\succ}(x)\eta(a),\ \ \ \eta(r_{\succ}(x)a)=r'_{\succ}(x)\eta(a),\label{h1}\\
  &\eta(l_{\prec}(x)a)=l'_{\prec}(x)\eta(a),\ \ \ \eta(r_{\prec}(x)a)=r'_{\prec}(x)\eta(a),\label{h2}\\
  	 	&\zeta(l_{\succ}(x)a)=x\succ \zeta(a)-\mu_{\succ}(a)x+\mu'_{\succ}(\eta(a))x,\label{h3}\\
  &\zeta(r_{\succ}(x)a)=\zeta(a)\succ x-\rho_{\succ}(a)x+\rho'_{\succ}(\eta(a))x,\label{h4}\\
  &\zeta(l_{\prec}(x)a)=x\prec \zeta(a)-\mu_{\prec}(x)a+\mu'_{\prec}(\eta(x))a,\label{h5}\\
    &\zeta(r_{\prec}(x)a)=\zeta(a)\prec x-\rho_{\prec}(a)x+\rho'_{\prec}(\eta(a))x,\label{h6}\\
  	 	&\eta (a\succ_Vb)=\eta(a)\succ'_V\eta(b)+l'_{\succ}(\zeta(a))\eta(b)+r'_{\succ}(\zeta(b))\eta(a),\label{h7}\\
  	 	&\zeta(a\succ_Vb)+\varpi_1(a,b)=\zeta(a)\succ\zeta(b)+\rho'_{\succ}(\eta(a))\zeta(b)+\mu'_{\succ}(\eta(b))\zeta(a)+\varpi_{1}'(\eta(a),\eta(b)),\label{h8}\\
  	 	&\eta (a\prec_Vb)=\eta(a)\prec'_V\eta(b)+l'_{\prec}(\zeta(a))\eta(b)+r'_{\prec}(\zeta(b))\eta(a),\label{h9}\\
  	 	&\zeta (a\prec_Vb)+\varpi_2(a,b)=\zeta(a)\prec\zeta(b)+\rho'_{\prec}(\eta(a))\zeta(b)+\mu'_{\prec}(\eta(b))\zeta(a)+\varpi_{2}'(\eta(a),\eta(b)).\label{h10}
  	 	\end{align}
  	 	 for any $x \in A$ and $a,b \in V$. Then there exists a bijection
	\begin{equation}\label{h11}\Gamma:\mathcal{P}\longrightarrow  \mathcal{H},~\Gamma (\zeta,\eta)=\psi_{(\zeta,\eta)},
\ \ \ \psi_{(\zeta,\eta)}(x,a)=(x+\zeta(a),\eta(a)),~\forall~x \in A, a\in V.\end{equation}
 Moreover,
the homomorphism $\psi=\psi_{(\zeta,\eta)}$ is an isomorphism if and only if $\eta$ is a bijection
 and $\psi=\psi_{(\zeta,\eta)}$ co-stabilizes $V$ if and only if $\eta=id_V$. 	 
 \end{lem}
 
\begin{proof}
First, we verify that $\Gamma$ is well-defined. For all $(\zeta,\eta)\in \mathcal{P}$ and $\psi=\psi_{(\zeta,\eta)}$, we have
\begin{equation*}\psi i(x)=\psi (x,0)=(x,0)=i'(x),~\forall~ x\in A,\end{equation*}
which means that $\psi$ stabilizes $A$. 
$\psi$ is a homomorphism of anti-dendriform algebras if and only if the following equation holds:
\begin{equation}\label{h12}\psi((x,a)\succeq (y,b))=\psi(x,a)\succeq' \psi(y,b),\ \ \ \psi((x,a)\preceq (y,b))=\psi(x,a)\preceq' \psi(y,b)\end{equation}
  for all $x,y \in A$ and $ a,b\in V$.
Using Eqs. \eqref{SE1}-\eqref{SE2} and \eqref{h11}, by direct computations, we get that
Eq. \eqref{h12} holds if and only if Eqs.\eqref{h1}-\eqref{h10} hold. Thus, $\Gamma$ is well-defined.

Second, we prove that $\Gamma$ is a bijection. It is obvious that $\Gamma$ is injective. 
For all $\psi\in  \mathcal{P}$, define linear maps
$\zeta_\psi:V\longrightarrow A$ and $\eta_\psi:V\longrightarrow V$ respectively by
\begin{equation}\zeta_\psi(a)=\pi_{A}\psi (0,a),~~\eta_\psi(a)=\pi'_{V}\psi (0,a),~\forall~a\in V.\end{equation}
Then \begin{align*}&\Gamma(\zeta_\psi,\eta_\psi)(x,a)=\psi_{(\zeta_\psi,\eta_\psi)}(x,a)\\=&(x+\pi_{A}\psi(0,a),\pi_{V}^{'}\psi(0,a))
\\=&\psi(x,0)+i'\pi_{A}\psi(0,a)+i_{V}\pi_{V}^{'}\psi(0,a)\\=&\psi(x,0)+(i'\pi_{A}+i_{V}\pi_{V}^{'})\psi(0,a)\\=&\psi(x,0)+\psi(0,a)
\\=&\psi(x,a),\end{align*}
which implies that $\Gamma$ is surjective. In all, $\Gamma$ is a bijection.
Suppose that $\eta:V\longrightarrow V$ is bijective. Then $\psi=\psi_{(\zeta,\eta)}$ is
 an isomorphism of anti-dendriform algebras
 with the inverse defined by
 $\psi^{-1}=\psi_{(\zeta,\eta)}^{-1}(x,a)=(x-\zeta\eta^{-1}(a),\eta^{-1}(a))$ for all $x\in A$ and $a\in V$.
On the other hand, assume that $\psi=\psi_{(\zeta,\eta)}$ is a bijection. Obviously, $\eta$ is surjective. 
Assume that $\eta(a)=0$, then
$\psi_{(\zeta,\eta)}(-\zeta(a),a)=(\zeta(a)-\zeta(a),\eta(a))=(0,0)$.
Since $\psi_{(\zeta,\eta)}$ is bijective, $a=0$. Thus, $\eta$ is injective. Therefore, $\eta$ is bijective.
Finally, it is obvious that $\psi=\psi_{(\zeta,\eta)}$ co-stabilizes $V$ if and only if $\eta=id_V$. 
 \end{proof}
  
\begin{defi}\label{01}
  	  Let $(A,\succ,\prec)$ be an anti-dendriform algebra, $V$ a vector space and $\Omega(A,V)=(l_{\succ},r_{\succ},l_{\prec},r_{\prec},\rho_{\succ},\mu_{\succ},\rho_{\prec},\mu_{\prec},
    \varpi_1,\varpi_2,\succ_V,\prec_V)$, 
    $\Omega(A,V)=(l_{\succ}',r_{\succ}',l_{\prec}',r_{\prec}',\rho_{\succ}',\mu_{\succ}',\rho_{\prec}',\mu_{\prec}',
    \varpi_{1}',\varpi_{2}',\succ'_V,\prec'_V)$
   be two extending structures of $(A,\succ,\prec)$ through $V$.
 If there is a pair $(\zeta,\eta)$ of linear maps, 
 where $\eta:V\longrightarrow V$ is a bijection and $\zeta:V\longrightarrow A$, satisfying 
  Eqs. \eqref{h1}-\eqref{h10}, then
  $\Omega(A,V)$ and $\Omega'(A,V)$ are called {\bf equivalent} and denote it by 
  $\Omega(A, V) \equiv \Omega'(A, V)$.
  Moreover, if $\eta=id_{V}$, $\Omega(A, V)$ and $ \Omega'(A, V)$ are called {\bf cohomologous}, which is 
denoted by $\Omega(A, V) \approx \Omega'(A, V)$.
  	\end{defi}
  
Now, we are ready to provide an answer for the ES-problem of anti-dendriform algebra by the following Theorem.
  
  	\begin{thm}\label{13}
  		Let $(A,\succ,\prec)$ be an anti-dendriform algebra, $E$ a vector space containing $A$ as a subspace and $V$ a complement of $A$ in $E$. 
  		\begin{enumerate}
  			\item Denote  $H^2_{p}(V, A)= \mathcal{A}(A, V)/ \equiv$. The map
  		\begin{align*}
  			\Phi:H^2_{p}(V, A) \longrightarrow EXT(E,A), \quad \overline{\Omega(A, V)} \longmapsto \overline{A \natural  V}
  		\end{align*}
  		is a bijection, where $\overline{\Omega(A, V)}$ and $\overline{A \natural  V}$ 
  are the equivalent class of $\Omega(A, V)$ and $A \natural  V$ under $\equiv$ respectively. 
  	\item Denote $H^2(V, A)=\mathcal{A}(A, V) / \approx$. The map
  		\begin{align*} 
  			\Psi:H^2(V,A) \longrightarrow EXT'(E,A), \quad [\Omega(A, V)] \longmapsto [A \natural  V]
  		\end{align*}
  		is a bijection, where $[\Omega(A, V)]$ and $[A \natural  V]$ 
  are the equivalent classes of $\Omega(A, V)$ and $A \natural  V$ under $\approx$ respectively,
  where $EXT(E,A)=\mathcal{G}(E,A)/\equiv$ and $EXT'(E,A)=\mathcal{G}'(E,A)/\approx$ are introduced in the Introduction.

  \end{enumerate}
  	\end{thm}
  	\begin{proof}
The statements can be verified by Proposition \ref{U1}, Theorem \ref{U2} and  Lemma \ref{63} .

  	 	\end{proof}

 %%%%%%%%%%%%%%%%%%%%%%%%%%%%%%%%%%%%%%%%%%%%%%%%%%%%%%%%%%%%%%%%%%%%%%%%%%%%%%%%%%%%%%%%%%%%%%%%%%%%%%% 
  \section{Crossed products, the GE-problem and non-abelian extensions}
 In the section, we first introduce the notion of 
  crossed products for anti-dendriform algebras, which
   is a special case of the unified products. Moreover, we give an answer to the GE-problem, which is a categorical dual
   of the ES-problem. Finally, we explore the inducibility problem of pairs of automorphisms associated with 
    non-abelian extensions of anti-dendriform algebras
and derive the Wells sequences under the context of anti-dendriform algebras.
  
  \begin{defi}
  		Let $(A,\succ,\prec)$ be an anti-dendriform algebra and $V$ a vector space. 
  A \textbf{pre-crossed datum of $(A,\succ,\prec)$ through $V$}
    is a system $C(A,V)=(l_{\succ},r_{\succ},l_{\prec},r_{\prec},\omega_1,\omega_2,\succ_V,\prec_V)$ 
    consisting of four linear maps 
  		\begin{align*}
  			l_{\succ},r_{\succ},l_{\prec},r_{\prec}: A \longrightarrow End(V), 
  			\end{align*}
  and four bilinear maps 
  \begin{align*} \omega_1,\omega_2: A\times A \longrightarrow V, \ \  \ \succ_V,\prec_V: V \times V \longrightarrow V.
  		\end{align*}
  		Denote by $A\sharp V $ the vector space $ A\oplus V$ 
  with the bilinear maps $\succeq,~\preceq: (A\oplus V)\times(A \oplus V) \longrightarrow A \oplus V$ 
  given as follows:
  \begin{align}
    	&(x,a)\succeq (y,b)=(x\succ y,\omega_1(x,y)+l_{\succ}(x)b+r_{\succ}(y)a+a\succ_{V}b),\label{A3}\\
    &(x,a)\preceq(y,b)=(x\prec y,\omega_2(x,y)+l_{\prec}(x)b+r_{\prec}(y)a+a\prec_{V}b),\label{A4}
    \end{align}
  for all $x,y\in A$ and $a,b\in V$.
 The object $A\sharp V $ is called the \textit{crossed product} of $(A,\succ,\prec)$ if 
  $(A\oplus V,\succeq,\preceq)$ is an anti-dendriform algebra. 
  In this case, the pre-crossed datum $C(A, V) $ is called a crossed system of
   $(A,\succ,\prec)$ through $V$. $\omega_1$ and $\omega_2$ are called cocycle of $C(A, V) $.
  Denote by $GC(A, V)$ the set of all crossed systems of $(A,\succ,\prec)$ through $V$.
  	\end{defi}
  
  \begin{pro}\label{CU1} Let $(A,\succ,\prec)$ be an anti-dendriform algebra, $V$ be a vector space and 
  $C(A,V)=(l_{\succ},r_{\succ},l_{\prec},r_{\prec},\omega_1,\omega_2,\succ_{V},\prec_{V})$ be a pre-crossed datum 
   of $(A,\succ,\prec)$
  through $V$. Then $A\sharp V$ is a crossed product of $(A,\succ,\prec)$ and $V$ if 
  and only if the following conditions hold for all $x,y,z\in A$ and $a,b\in V$,
 \begin{align}\label{C1}&l_{\succ}(x)\omega_{1}(y,z)a+\omega_1(x,y\succ z)=-r_{\succ}(z)\omega(x,y)-\omega_{1}(x\cdot y,z)
   \\\notag=&-l_{\prec}(x)\omega(y,z)-\omega_{2}(x,y\cdot z)=r_{\prec}(z)\omega_2(x,y)+\omega_{2}(x\prec y, z),\end{align}
 \begin{align} \label{C2} l_{\succ}(x)l_{\succ}(y)a=l_{\prec}(x\prec y)a+\omega_{2}(x,y)\prec_{V}a=
 -l_{\succ}(x\cdot y)a-\omega(x,y)\succ_{V}a=-l_{\prec}(x)l_{\cdot}(y)a, \end{align}
   \begin{align} \label{C3}l_{\succ}(x)r_{\succ}(y)a=r_{\prec}(y)l_{\prec}(x)a=
 -r_{\succ}(y)l_{\cdot} (x)a=-l_{\prec}(x)r_{\cdot}(y)a,\end{align}
   \begin{align} \label{C4}r_{\succ}(x\succ y)a+a\succ_{V}\omega_1(x,y)=r_{\prec}(y)r_{\prec} (x)a
   =-r_{\prec}(x\cdot y)a-a\succ_{V}\omega(x,y)=-r_{\succ}(y)r_{\cdot} (x)a,\end{align}
   \begin{align} \label{C5} l_{\succ}(x)( a\succ_{V}b)=(l_{\prec}(x)a)\prec_{V}b=-(l_{\cdot}(x)a)\succ_{V}b=-l_{\prec}(x)(a\cdot_{V}b),\end{align}
   \begin{align}\label{C6}a\succ_{V}(l_{\succ}(x)b)=(r_{\prec}(x)a)\prec_{V}b=-(r_{\cdot}(x)a)\succ_{V}b=-a\prec_{V}(l_{\cdot}(x)b)
   ,\end{align}
   \begin{align} \label{C7}a\succ_{V}(r_{\succ}(x)b)=r_{\prec}(x)(a\prec_{V}b)=-r_{\succ}(x)(a\cdot_{V}b)=-a\prec_{V}(r_{\cdot}(x)b),\end{align}
    \begin{align}\label{C8}r_{\prec}(z)\omega_1(x,y)+\omega_2(x\succ y,z)=l_{\succ}(x)\omega_2(y,z)+\omega_1(x,y\prec z),\end{align}
    \begin{align}\label{C9}l_{\prec}(x\succ y)a+\omega_1(x,y)\prec_{V}a=l_{\succ}(x)l_{\prec}(y)a, 
    \ \ \ r_{\prec}(y)l_{\succ} (x)a=l_{\succ}(x)r_{\prec}(y)a,\end{align}
    \begin{align}\label{C10}r_{\prec}(y)r_{\succ} (x)a=r_{\succ}(x\prec y)a+a\succ_{V}\omega_2(x,y), \ \ \ 
  (l_{\succ}(x)a)\prec_{V}b=l_{\succ}(x)(a\prec_{V}b),\end{align}
    \begin{align}\label{C11} (r_{\succ}(x)a)\prec_{V}b=a\succ_{V}(l_{\prec}(x)b), \ \ \  r_{\prec}(x)(a\succ_{V}b)=a\succ_{V}(r_{\prec}(x)b),\end{align}
      \begin{align} \label{C12} \text{$(V,\succ_{V},\prec_{V})$ is an anti-dendriform algebra.}\end{align}
where $l_{\cdot}=l_{\succ}+l_{\prec},~~r_{\cdot}=r_{\succ}+r_{\prec},~~\omega=\omega_1+\omega_2$
 and $\cdot_{V}=\succ_{V}+\prec_{V},\cdot=\succ+\prec$.
  \end{pro}
  
\begin{proof}  The proof follows the same argument as Proposition \ref{U1}.
  
  \end{proof}

\begin{rmk}
Let $(A,\succ,\prec)$ be an anti-dendriform algebra and $V$ a vector space. Suppose that
 $\Omega(A,V)=(l_{\succ},r_{\succ},l_{\prec},r_{\prec},\rho_{\succ},\mu_{\succ},\rho_{\prec},\mu_{\prec},
    \varpi_1,\varpi_2,\succ_V,\prec_V)$
is an extending datum of $(A,\succ,\prec)$ through $V$. When
 $l_{\succ},r_{\succ},l_{\prec},r_{\prec}$ are trivial maps, 
 then $\Omega(A,V)=(\rho_{\succ},\mu_{\succ},\rho_{\prec},\mu_{\prec},
    \varpi_1,\varpi_2,\succ_V,\prec_V)$ is also called pre-crossed datum of $(A,\succ,\prec)$ through $V$,
 where the anti-dendriform algebra structure $(\succeq,\preceq)$ on $A\oplus V$ is given by
 \begin{align*}
 & (x,a)\succeq (y,b)=(x\succ y+\rho_{\succ}(a)y+\mu_{\succ}(b)x+\varpi_1(a,b),a\succ_V b)),\\
    & (x,a)\preceq(y,b)=(x\prec y+\rho_{\prec}(a)y+\mu_{\prec}(b)x+\varpi_2(a,b),a\prec_V b)),~~\forall~~x,y \in A,a,b \in V.
\end{align*}
 In this case, $\Omega(A,V)=(\rho_{\succ},\mu_{\succ},\rho_{\prec},\mu_{\prec}, \varpi_1,\varpi_2,\succ_V,\prec_V)$ is 
also called a {\bf crossed system } of the anti-dendriform algebra $(A,\succ,\prec)$ through $V$,
and the associated unified product $A\natural V$ is called the {\bf crossed product}. One can find the crossed product for Poisson algebras
in the same sense in \cite{7}. 

\end{rmk}

\begin{thm}\label{CU2} 
    Let $(A,\succ,\prec)$ be an anti-dendriform algebra, $E$ a vector space and $\pi:E\longrightarrow A$ a sujective map
    with $V=Ker(\pi)$. Then any anti-dendriform algebra structure $(\succeq,\preceq)$ 
    which can be given on $E$ such that $\pi:E\longrightarrow A$ is a homomorphism of anti-dendriform algebras is isomorphic
     to a crossed product $A\sharp V$. Furthermore, the isomorphism of anti-dendriform algebras 
  $(E,\succeq,\preceq) \cong A\sharp V$ can be chosen such that it stabilizes $V$ and co-stabilizes $A$.
  \end{thm}
  \begin{proof} Let the linear map $s:A\longrightarrow E$ be a section of $\pi$, such that $\pi s=I_A$.
  Define linear maps
\begin{align*}l_{\succ},r_{\succ},l_{\prec},r_{\prec} : A \longrightarrow End(V)\end{align*}
  and bilinear maps 
  \begin{align*} \omega_1,\omega_2: A\times A \longrightarrow V\end{align*}
  respectively by
  \begin{equation}\label{cm3}l_{\succ}(x)a=s(x)\succ_{E} a,\ \ \ r_{\succ}(x)a=a\succ_{E} s(x),\end{equation} 
  \begin{equation}\label{cm4}l_{\prec}(x)a=s(x)\prec_{E} a,\ \ \ r_{\prec}(x)a=a\prec_{E} s(x),\end{equation} 
  \begin{equation}\label{cm5}\omega_1(x,y)=s(x)\succ_{E} s(y)-s(x\succ y),\ \ \omega_2(x,y)=s(x)\prec_{E} s(y)-s(x\prec y),\end{equation} 
  for all $a,b\in V$ and $x,y\in A$. Then  
  \begin{equation}\label{cm6}\varphi:A\oplus V\longrightarrow E,\ \ \varphi(x,a)=a+s(x),\ \forall~x\in A,~a\in V\end{equation}  
  is a linear isomorphism with the inverse  
  \begin{equation}\label{cm7}\varphi^{-1}:E\longrightarrow A\oplus V,\ \ \varphi^{-1}(e)=(\pi(e),e-s\pi(e)),\ \forall~e\in E.\end{equation} 
 We need to state that the anti-dendriform structure $(\succeq,\preceq)$ that can be given on $A\oplus V$ such that 
  $\varphi$ is an isomorphism of anti-dendriform algebras identifies with the one defined by Eqs. \eqref{A3}-\eqref{A4}.
   Indeed, since $\pi$ is a homomorphism of anti-dendriform algebras, for all $x,y\in A$ and $a,b\in V$, we have
 \begin{align}\label{cm8}&\pi(s(x)\succ_{E} s(y)+s(x)\succ_{E}b+a\succ_{E} b+a\succ_{E} s(y))
\\  \notag=& \pi s(x)\succ \pi s(y)+\pi s(x)\succ \pi(b)+\pi(a)\succ \pi(b)+\pi(a)\succ_{E} \pi s(y)
 \\  \notag=&x\succ y
 \end{align}
 and
  \begin{equation}\label{cm9}s\pi(s(x)\succ_{E} s(y)+s(x)\succ_{E}b+a\succ_{E} b+a\succ_{E} s(y))=s(x\succ y).\end{equation}
 By Eqs. \eqref{cm3}-\eqref{cm9}, we obtain
 \begin{align*}
 	&(x,a)\succeq (y,b)=\varphi^{-1}(\varphi(x,a)\succ_{E} \varphi(y,b))
 \\=&\varphi^{-1}((a+s(x))\succ_{E} (b+s(y)))\\=&\varphi^{-1}(a\succ_{E} s(y)+s(x)\succ_{E}s(y)+s(x)\succ_{E} b+a\succ_{E} b)
 \\=&(x\succ y,l_{\succ}(x)b+r_{\succ}(y)a+\omega_1(x,y)+a\succ_V b).
 \end{align*}
 Analogously, 
 \begin{equation*}(x,a)\preceq (y,b)=(x\prec y,l_{\prec}(x)b+r_{\prec}(y)a+\omega_2(x,y)+a\prec_V b).\end{equation*}
In addition, it is easy to check that the following diagram is commutative:
    \begin{equation*} \xymatrix{
  0 \ar[r] & V\ar@{=}[d] \ar[r]^{i_V} & A\sharp V\ar[d]_{\varphi} \ar[r]^{\pi_A} & A \ar@{=}[d] \ar[r] & 0\\
 0 \ar[r] & V \ar[r]^{i} & E \ar[r]^{\pi} & A  \ar[r] & 0 .}\end{equation*}
  We complete the proof. 
 
  	 \end{proof}
  	 
 By Theorem \ref{CU2}, any anti-dendriform algebra structure $(\succeq,\preceq)$ on a vector space $E$ such that
 $\pi:E\longrightarrow A$ is a homomorphism of
anti-dendriform algebras is isomorphic to a crossed product of $A$ 
 through $V=Ker(\pi)$. Hence the GE-problem reduces to the classification of all 
 crossed products $A\sharp V$, associated to all crossed system of $(A,\succ,\prec)$ by $V$. 
  	
  In the sequel, we will classify all crossed products $A\sharp V$.
  	 
\begin{defi}\label{01}
  	  Let $(A,\succ,\prec)$ be an anti-dendriform algebra and $V$ a vector space. Suppose that
   $C(A,V)=(l_{\succ},r_{\succ},l_{\prec},r_{\prec},
    \omega_1,\omega_2,\succ_V,\prec_V)$, 
    $C'(A,V)=(l_{\succ}',r_{\succ}',l_{\prec}',r_{\prec}',\omega_{1}',\omega_{2}',\succ^{'}_V,\prec^{'}_V)$
   are two crossed systems of $(A,\succ,\prec)$ through $V$.
 If there is a linear map $\zeta:A\longrightarrow V$, satisfying the following conditions:
\begin{align}
  &l_{\prec}(x)a=l'_{\prec}(x)a+\zeta(x)\prec^{'}_Va,~~l_{\succ}(x)a=l'_{\succ}(x)a+\zeta(x)\succ^{'}_Va,\label{N1}\\
  	 	&r_{\prec}(x)a=r'_{\prec}(x)a+a\prec^{'}_V\zeta(x),~~r_{\succ}(x)a=r'_{\succ}(x)a+a\succ^{'}_V\zeta(x),\label{N2}\\
  &\omega_1(x,y)+\zeta(x\succ y) =\omega_1^{'}(x,y)+l^{'}_{\succ}(x)\zeta(y)+r^{'}_{\succ}(y)\zeta(x)+\zeta(x)\succ^{'}_V\zeta(y)
  ,\label{N3}\\
  &\omega_2(x,y)+\zeta(x\prec y) =\omega_2^{'}(x,y)+l^{'}_{\prec}(x)\zeta(y)+r^{'}_{\prec}(y)\zeta(x)+\zeta(x)\prec^{'}_V\zeta(y)
  ,\label{N4}\\
  &a\succ_Vb=a\succ^{'}_Vb,~~a\prec_Vb=a\prec^{'}_Vb,\label{N5}
  	 	\end{align}
  	 	 for any $x \in A$ and $a,b \in V$. Then
  $C(A,V)$ and $C'(A,V)$ are called {\bf cohomologous}, which is 
denoted by $C(A, V) \approx C'(A, V)$.
  	\end{defi}
  
Now, we are ready to provide an answer for the GE-problem of anti-dendriform algebras by the following Theorem.
  
  	\begin{thm}\label{GE}
  		Let $(A,\succ,\prec)$ be an anti-dendriform algebra, $E$ a vector space and
   $\pi:E\longrightarrow A$ an epimorphism
    of vector spaces with $V=Ker(\pi)$. Then $\approx $ is an equivalence relation on the set $GC(A,V)$ of all
    crossed systems of $(A,\succ,\prec)$ by $V$. Denote $GH^{2}(A,V)=GC(A,V)/\approx$, then the map
$\Theta:GH^{2}(A,V) \longrightarrow  \mathcal{C}_{\pi}(E,A)/\approx$ is a bijection, where $\Theta$ assigns 
the equivalent class of $C(A, V)$ to the equivalent class of $A\sharp V$ and $\mathcal{C}_{\pi}(E,A)$ appeared in the Introduction. 
\end{thm}
  	\begin{proof}
The conclusion can be verified by Proposition \ref{CU1} and Theorem \ref{CU2} .
  	 	\end{proof}  

Calculating $GH^{2}(A,V)$ is a highly nontrivial problem. We connect it with the classical extension problem.

Let $A$ and $B$ be two anti-dendriform algebras.
 A {\bf non-abelian extension} of $A$ by $B$ is an anti-dendriform algebra $E$ equipped with a 
 short exact sequence of anti-dendriform algebras:
	\begin{align*}
		\mathcal{E}:0\longrightarrow B\stackrel{i}{\longrightarrow} E\stackrel{p}{\longrightarrow}A\longrightarrow0.
	\end{align*}
When $B$ is abelian, the extension is called an {\bf abelian extension} of $A$ by $B$. A {\bf section} of $p$ is a linear map $s: A
\rightarrow E $ such that $ps = I_{A}$.

Let $ E_1$ and $E_2$
be two non-abelian extensions of $A$ by $B$. They are said to be 
{\bf cohomologous} if there is an isomorphism of anti-dendriform algebras
$\varphi:E_1\longrightarrow E_2$ such that
the following commutative diagram holds:
 \begin{equation}\label{E1} \xymatrix{
  0 \ar[r] & B\ar@{=}[d] \ar[r]^{i_1} & E_1\ar[d]_{\varphi} \ar[r]^{p_1} & A \ar@{=}[d] \ar[r] & 0\\
 0 \ar[r] & B \ar[r]^{i_2} & E_2 \ar[r]^{p_2} & A  \ar[r] & 0
 .}\end{equation}
It is natural to describe all anti-dendriform algebras $E$ which are extensions 
of $A$ by $B$. This is the classical extension problem. 
 
\begin{thm} \label{Es}
Let $(A,\succ_A,\prec_A)$ and $(V,\succ_{V},\prec_{V})$ be two anti-dendriform algebras.
Assume that
 \begin{equation*}0\longrightarrow V\stackrel{i}{\longrightarrow} E\stackrel{p}{\longrightarrow} A\longrightarrow0\end{equation*} 
  is
an extension of $(V,\succ_{V},\prec_{V})$ by $(A,\succ_A,\prec_A)$ with a section $s$ of $p$. Then any 
extension of $A$ by $V$ is cohomologous to a crossed product extension:
  \begin{align}\label{E2}
	0\longrightarrow V\stackrel{i}{\longrightarrow} A\sharp V \stackrel{\pi}{\longrightarrow}A\longrightarrow0.
	\end{align}
\end{thm}
\begin{proof}
It can be obtained by Theorem \ref{CU2}.
\end{proof}

Theorem \ref{CU2}, Theorem \ref{GE} and Theorem \ref{Es} 
indicate that computing the classifying object $ \mathcal{C}_{\pi}(E,A)$ reduces to the classification
of all crossed extension of (\ref{E2}).

Let $C(A,V)=(l_{\succ},r_{\succ},l_{\prec},r_{\prec},\omega_1,\omega_2,\succ_V,\prec_V)$ be a crossed system of
$(A,\succ_{A},\prec_{A})$ through $V$. For a fixed anti-dendriform algebra $(V,\succ_{V},\prec_{V})$,
let $C_{l} (A,V)=(l_{\succ},r_{\succ},l_{\prec},r_{\prec},\omega_1,\omega_2)$
 satisfying Eqs. \eqref{C1}-\eqref{C11},
  which is called the 
{\bf local crossed system} of $(A,\succ_{A},\prec_{A})$ by $(V,\succ_{V},\prec_{V})$ or
{\bf non-abelian 2-cocycle } of $(A,\succ_{A},\prec_{A})$ with values in $(V,\succ_{V},\prec_{V})$.
Obviously, $(A\sharp V,\succeq,\preceq)$ is an anti-dendriform algebra with $(\succeq,\preceq)$ given by
Eqs. \eqref{A3}-\eqref{A4}.

Non-abelian 2-cocycles $C_{l} (A,V)=(l_{\succ},r_{\succ},l_{\prec},r_{\prec},\omega_1,\omega_{2})$ and 
$C_{l} ^{'}(A,V)=(l'_{\succ},r'_{\succ},l'_{\prec},r'_{\prec},\omega'_1,\omega'_{2})$ 
are called {\bf cohomologous}, denoted by 
$C_{l} (A,V)\approx C_{l}^{'}$
$ (A,V)$, if there is a linear map $\zeta:A \longrightarrow V$ satisfying Eqs. \eqref{N1}-\eqref{N4}.
Denote by $Z^{2}_{nab}(A,V)$ the set of all non-abelian 2-cocycles $C_{l} (A,V)$ of $(A,\succ_{A},\prec_{A})$ with values 
in $(V,\succ_{V},\prec_{V})$. The quotient of $Z^{2}_{nab}(A,V)$ by the above
equivalence relation is denoted by $H^{2}_{nab}(A,V)$, which is called the non-abelian second cohomology group
of $A$ in $V$. It is easy to check that $H_{nab}^{2}(A,V)=H^{2}(A,V)$ when
$A$ is an abelian anti-dendriform algebra, where $H^{2}(A,V)$ is the second cohomology group
of $A$ in $V$. 

Analogous to Theorem \ref{GE}, we have

\begin{thm} \label{101}
Let $(A,\succ_{A},\prec_{A})$ be an anti-dendriform algebra.
\begin{enumerate}
	\item If $(V,\succ_{V},\prec_{V})$ is a fixed anti-dendriform algebra. The map
	\begin{equation*}\Phi:H^{2}_{nab}(A,V)\longrightarrow   \mathcal{C}_{\pi}(E,A)/\approx,\ \ \ \Phi([Z^{2}_{nab}(A,V)])=[A\sharp V]\end{equation*}
	is a bijection, where $[Z^{2}_{nab}(A,V)]$ is the equivalent class of $Z^{2}_{nab}(V,A)$
	and $[A\sharp V]$ is the equivalent class of the crossed extension defined by (\ref{E2}).
	\item Let $E$ a vector space containing $A$ as an anti-dendriform ideal and $V$ a vector subspace of $E$.
	Then 
	\begin{equation*} GH^{2}(A,V)=\bigsqcup_{(V,\succ_{V},\prec_{V})}H^{2}_{nab}(A,V).\end{equation*}
\end{enumerate}
\end{thm}

In general, it is not easy to calculate $GH^{2}(V,A)$. When $A$ and $V$ are abelian anti-dendriform algebras, 
we denote by $A_0$ and $V_{0}$ respectively.
Then $GH^{2}(A_0,V_0)=H^{2}_{nab}(A_0,V_0)$.

\begin{ex}
Let $k_0$ be an abelian anti-dendriform algebra of dimension $n$ with basis $E_{n+1}$
and
$k_0^{n}$ an n-dimensional abelian anti-dendriform algebra with basis $\{E_1,E_2,\cdot\cdot\cdot,E_n \}$. Suppose that
 $C_{l} (k_0,k_0^{n})=(l_{\succ},r_{\succ},l_{\prec},r_{\prec},\omega_1,\omega_2)$ 
is the local system of $k_0$ through $k_0^{n}$ and 
$(k_0\sharp k_0^{n},\succeq,\preceq)$ is the corresponding 
crossed product. Assume that
\begin{align*}
		&E_{n+1}\succeq E_{i}:=l_{\succ}(E_{n+1})E_i=\sum_{j = 1}^n a_{ji}E_{j}, \ \ \ 
E_{n+1}\preceq E_{i}:=l_{\prec}(E_{n+1})E_i=\sum_{j = 1}^n c_{ji}E_{j},
\\
&E_{i}\succeq E_{n+1}:=r_{\succ}(E_{n+1})E_i=\sum_{j = 1}^n b_{ji}E_{j},
\ \ \ E_{i}\preceq E_{n+1}:=r_{\prec}(E_{n+1})E_i=\sum_{j = 1}^n d_{ji}E_{j},\\
&E_{n+1}\succeq E_{n+1}:=\omega_1(E_{n+1},E_{n+1})=\sum_{j = 1}^n \theta_{0j}E_{j}, \ \ \
		E_{n+1}\preceq E_{n+1}:=\omega_2(E_{n+1},E_{n+1})=\sum_{j = 1}^n \epsilon_{0j}E_{j},
	\end{align*}	
Combining Eqs.~(\ref{C1})-(\ref{C11}), we have
\begin{align*}
	&A^2=-C(A+C)=0,\ \ \ AB=DC=-B(A+C)=-C(B+D)B,\\&
D^2=-B(D+B)=0, \ \  \ AC=0, \ \  \ DB=0,\ \ \ AD=DA,\\
	&A\theta_0=-B(\theta_{0}+\epsilon_0)=-C(\theta_{0}+\epsilon_0)=-D\epsilon_{0},\ \ \ D\theta_0=A\epsilon_0.
\end{align*}
where $A=(a_{ij})_{n\times n},~B=(b_{ij})_{n\times n},~C=(c_{ij})_{n\times n},~D=(d_{ij})_{n\times n}\in M_{n}(k)$ and
$\theta_{0}=(\theta_{0j})_{n\times 1},\epsilon_{0}=(\epsilon_{0j})_{n\times 1}\in k^{n}$ are matrices over the field $k$.

Two six-tuples $(A,B,C,D,\theta_{0},\epsilon_{0})$ and $(A',B',C',D',\theta'_{0},\epsilon'_{0})$ are cohomologous,
denoted by $(A,B,C,D,\theta_{0},\epsilon_{0})\approx(A',B',C',D',\theta'_{0},\epsilon'_{0})$,
if and only if $A=A'$, $B=B'$, $C=C'$, $D=D'$ and there exists $w\in k^{n}$ 
such that $\theta_{0}-\theta'_{0}=(A+B)w,~\epsilon_{0}-\epsilon'_{0}=(C+D)w$.
Let $\mathfrak{C}(n)$ be the set of all six-tuples $(A,B,C,D,\theta_{0},\epsilon_{0})$. 
Then, $GH^{2}(k_{0},k^{n}_{0})\cong \mathfrak{C}(n)/\approx$. We explain it in detail in the following.

Assume that $Z^{2}_{nab}(k_{0},k^{n}_{0})$ is the set of all local systems of
 $k^{n}_{0}$ through $k_{0}$. 
 Denote by $\mathcal{L}$ the set of all six-tuples $(\alpha,\beta,\gamma,\lambda,\upsilon_{0},\theta_{0})$,
 where $\alpha,\beta,\gamma,\lambda:k^{n}\longrightarrow k^{n}$ are three linear maps 
and 
 $\theta_{0},\epsilon_{0}\in k^{n}$ are vectors, satisfying the compatibility conditions: 
 \begin{align*}
	&\alpha^2=-\gamma(\alpha+\gamma)=0,\ \ \ \alpha\beta=\lambda\gamma=-\beta(\alpha+\gamma)=-\gamma(\beta+\lambda)\beta,\\&
\lambda^2=-\beta(\lambda+\beta)=0, \ \  \ \alpha\gamma=0, \ \  \ \lambda\beta=0,\ \ \ \alpha\lambda=\lambda\alpha,\\
	&\alpha(\theta_0)=-\beta(\theta_{0}+\epsilon_0)=-\gamma(\theta_{0}+\epsilon_0)=-\lambda\epsilon_{0},\ \ \ \lambda\theta_0=\alpha\epsilon_0.
\end{align*}
It is easy to check that there is a bijection between the sets $Z^{2}_{nab}(k_{0},k^{n}_{0})$ and $\mathcal{L}$.
Thus, we can obtain a local system $(l_{\succ},r_{\succ},l_{\prec},r_{\prec},\omega_1,\omega_2)$ corresponding to 
$(\alpha,\beta,\gamma,\lambda,\theta_{0},\epsilon_{0})$ as follows:
	\begin{align*}
		&l_{\succ}(x) a:=x\alpha(a), \ \ \ l_{\prec}(x)a:=x\gamma(a),\\
		&r_{\succ}(x) a:=x\beta(a), \ \ \ r_{\prec}(x)a:=x\lambda(a), 
\\&\omega_1(x,y):=xy\theta_{0}, \ \ \ \omega_2(x,y):=xy\epsilon_{0},
	\end{align*} 
	for all $x,y\in k$ and $a\in k^{n}$. Assume that $\{e_{1},e_{2},\cdots,e_{n}\}$ is the canonical basis of $k^{n}$.
Let $A,B,C,D$ be the matrices associated to $\alpha,\beta,\gamma$ and $\lambda$ with respect to this basis respectively.
 If we take the vectors $E_{1}=(e_{1},0),\cdots,E_{n}=(e_{n},0)$ and $E_{n+1}=(0,1)$ as a basis in $k^{n}\times k$, then the anti-dendriform algebra $k^{n+1}_{(A,B,C,D,\theta_{0},\epsilon_{0})}$ is just the crossed product $k_{0}\sharp k_{0}^{n}$ with the structures given by 
  \begin{align*}
		&(x,a)\succeq (y,b)=(0,l_{\succ}(x)b+r_{\succ}(y)a+\omega_1(x,y)),\\
		&(x,a)\preceq (y,b)=(0,l_{\prec}(x)b+r_{\prec}(y)a+\omega_2(x,y)),
	\end{align*}
  for all $a,b\in k_{0}^{n}$ and $x,y\in k_{0}$. Therefore, $GH^{2}(k_{0},k^{n}_{0})\cong \mathfrak{C}(n)/\approx$. 
\end{ex}
 
  %%%%%%%%%%%%%%%%%%%%%%%%%%%%%%%%%%%%%%%%%%%%%%%%%%%%%%%%%%%%%%%%%%%%%%%%%%%%%%%%%%%%%%%%%%%%%%%%%%%%%%%
\subsection{Inducibility of pairs of anti-dendriform algebra automorphisms}
In this section, we study inducibility of pairs of anti-dendriform algebra
automorphisms and characterize them by equivalent conditions.

  Let $(A,\succ_A,\prec_A)$ and $(B,\succ_B,\prec_B)$ be
two anti-dendriform algebras. Assume that
$\mathcal{E}:0\longrightarrow B\stackrel{i}{\longrightarrow} E\stackrel{p}{\longrightarrow} A\longrightarrow0$ is
a non-abelian extension of $(A,\succ_A,\prec_A)$ by $(B,\succ_B,\prec_B)$ with a section $s$ of $p$.
Define linear maps $ l_{\succ,s},r_{\succ,s},l_{\prec,s},r_{\prec,s}:A\longrightarrow \hbox{End}(V)$
and bilinear maps $\omega_{1,s},\omega_{2,s}:A\times A\longrightarrow V$
respectively by
\begin{equation}\label{cm14}l_{\succ,s}(x)a=s(x)\succ_E a,\ \ \ r_{\succ,s}(x)a=a\succ_E s(x),\end{equation}
\begin{equation}\label{cm15}l_{\prec,s}(x)a=s(x)\prec_E a,\ \ \ r_{\prec,s}(x)a=a\prec_E s(x),\end{equation}
\begin{equation}\label{cm16}\omega_{1,s}(x,y)=s(x)\succ_E  s(y)-s(x\succ y), \ \ \ \omega_{2,s}(x,y)=s(x)\prec_E s(y)-s(x\prec y)\end{equation}
for all $x,y\in A,a\in B.$ According to the proof of Theorem \ref{CU2}, we have

\begin{pro} With the above notations, $(l_{\succ,s},r_{\succ,s},l_{\prec,s},r_{\prec,s},\omega_{1,s},\omega_{2,s})$ is
a non-abelian 2-cocycle of $(A,\succ_A,\prec_A)$ by $(B,\succ_B,\prec_B)$. We call it the
non-abelian 2-cocycle corresponding to the non-abelian extension $\mathcal{E}$ induced by $s$.
Furthermore, there is an anti-dendriform algebra structure $(\succ_{\omega_{1,s}},\prec_{\omega_{2,s}})$ on $A\oplus B$, 
denote this anti-dendriform algebra by $(A\sharp_{\omega_s}B,\succ_{\omega_{1,s}},\prec_{\omega_{2,s}})$.
 Denote $(l_{\succ,s},r_{\succ,s},l_{\prec,s},r_{\prec,s},\omega_{1,s},\omega_{2,s})$
  simply by  $(l_{\succ},l_{\succ},l_{\prec},r_{\prec},\omega_{1},\omega_{2})$ 
  and denote $(A\sharp_{\omega_s}B,\succ_{\omega_{1,s}},\prec_{\omega_{2,s}})$ 
  by $(A\sharp_{\omega}B,\succ_{\omega_{1}},\prec_{\omega_{2}})$
   when the section used is clear in the context.\end{pro}

\begin{lem}\label{CM} Let $(l_{\succ},r_{\succ},l_{\prec},r_{\prec},
\omega_1,\omega_2)$ and $(l'_{\succ},r'_{\succ},l'_{\prec},r'_{\prec},
\omega'_1,\omega'_2)$ be the non-abelian 2-cocycles
 corresponding to the non-abelian extension
  $\mathcal{E}:0\longrightarrow B\stackrel{i}{\longrightarrow} E\stackrel{p}{\longrightarrow}A\longrightarrow0$
induced by the sections $s,s'$ of $p$.
 Then $(l_{\succ},r_{\succ},l_{\prec},r_{\prec},
\omega_1,\omega_2)$ and $(l'_{\succ},r'_{\succ},l'_{\prec},r'_{\prec},
\omega'_1,\omega'_2)$
 are cohomologous. In other words, the equivalent classes of non-abelian 2-cocycles corresponding to
 a non-abelian extension
induced by sections are independent on the choice of sections. Moreover, $A\sharp_{\omega}B$ and $A\sharp_{\omega'}B$
  are isomorphic.\end{lem}

\begin{proof}
 Assume that $(l_{\succ},r_{\succ},l_{\prec},r_{\prec},
\omega_1,\omega_2)$, $(l'_{\succ},r'_{\succ},l'_{\prec},r'_{\prec},\omega'_1,\omega'_2)$ are non-abelian
2-cocycles induced by sections $s$ and $s'$ respectively.
 Define a linear map $\varphi:A
\rightarrow B$ by $\varphi(x)=s(x)-s'(x)$. Since
$p\varphi(x)=ps(x)-ps'(x)=0$, $\varphi$ is well-defined. By
Eqs.~{\rm (\ref{cm14})}-{\rm (\ref{cm16})},
\begin{align*}&w_1(x,y)-w^{'}{_1}(x,y)=s(x)\succ_{E}s(y)-s(x\succ_A y)-s'(x)\succ_{E}s'(y)+s'(x\succ_A y)
\\=&s(x)\succ_{E}s(y)-s'(x)\succ_{E}s(y)+s'(x)\succ_{E}s(y)-s'(x)\succ_{E}s'(y)-s(x\succ_A y)+s'(x\succ_A y)
\\=&\varphi(x)\succ_{E}s(y)+s'(x)\succ_{E}\varphi(y)-\varphi(x\succ_A y)
\\=&\varphi(x)\succ_{E}\varphi(y)+\varphi(x)\succ_{E}s'(y)+l^{'}_{\succ}(x)\varphi(y)-\varphi(x\succ_A y)
\\=&\varphi(x)\succ_{E}\varphi(y)+r^{'}_{\succ}(y)\varphi(x)+l^{'}_{\succ}(x)\varphi(y)-\varphi(x\succ_A y),\end{align*} 
which yields that
 Eq.~(\ref{N3}) holds. By the same token, Eqs.~(\ref{N1})-(\ref{N2}) and (\ref{N4}) hold. Thus, $(l_{\succ},r_{\succ},l_{\prec},r_{\prec},
\omega_1,\omega_2)$ and $(l'_{\succ},r'_{\succ},l'_{\prec},r'_{\prec},
\omega'_1,\omega'_2)$ are cohomologous via the linear map $\varphi$.
Define a linear isomorphism
\begin{equation*}\phi:A\sharp_{\omega}B\longrightarrow A\sharp_{\omega'}B,\ \ \ \phi(x,a)=(x,\varphi(x)+a),~~\forall~x\in A,a\in B\end{equation*} 
with the inverse 
\begin{equation*}\phi^{-1}:A\sharp_{\omega'}B\longrightarrow A\sharp_{\omega}B,\ \ \ \phi(x,a)=(x,a-\varphi(x)),~~\forall~x\in A,a\in B\end{equation*} 
For all $x,y\in A,a,b\in V$, we have
\begin{align*}&\phi((x,a)\succ_{\omega_1}(y,b))
\\=&\phi(x\succ_A y,\omega_1(x,y)+l_{\succ}(x)b+r_{\succ}(y)a+a\succ_B b)
\\=&(x\succ_A y,\omega_1(x,y)+l_{\succ}(x)b+r_{\succ}(y)a+a\succ_B b+\varphi(x\succ_A y))
,\end{align*}
and
\begin{align*}&\phi(x,a)\succ_{\omega_{1}'}\phi(y,b)
\\=&(x,\varphi(x)+a)\succ_{\omega_{1}'}(y,\varphi(y)+b)
\\=&(x\succ_A y,\omega'_1(x,y)+l'_{\succ}(x)b+r'_{\succ}(y)a+l'_{\succ}(x)\varphi(y)+r'_{\succ}(y)\varphi(x)
+a\succ_B\varphi(y)+\varphi(x)\succ_Bb\\&+\varphi(x)\succ_B\varphi(y)
+a\succ_B b+\varphi(x\succ_A y))
\end{align*}
Thanks to Eqs.~(\ref{h4})-(\ref{h7}), we get that
$$\phi((x,a)\succ_{\omega_1}(y,b))=\phi(x,a)\succ_{\omega_{1}'}\phi(y,b).$$
Analogously, $\phi((x,a)\prec_{\omega_1}(y,b))=\phi(x,a)\prec_{\omega_{1}'}\phi(y,b).$
Hence, $\phi$ is an isomorphism of anti-dendriform
algebras. 
This finishes the proof.
\end{proof}

Let $(A,\succ_A,\prec_A)$ and $(B,\succ_B,\prec_B)$ be
two anti-dendriform algebras. Assume that
$\mathcal{E}:0\longrightarrow B\stackrel{i}{\longrightarrow} E\stackrel{p}{\longrightarrow} A\longrightarrow0$ is
a non-abelian extension of $(A,\succ_A,\prec_A)$ by $(B,\succ_B,\prec_B)$ with a section $s$ of $p$.
Analogous to the case of Lie algebras\cite{100}, we can define a group homomorphism
$$K:\mathrm{Aut}_{B} (E)\longrightarrow \mathrm{Aut}
(A)\times \mathrm{Aut} (B),~~K(\gamma)=(p\gamma s,\gamma|_{B}),$$ where
$$\mathrm{Aut}_{B}
(E)=\{\gamma\in \mathrm{Aut} (E)\mid \gamma
(B)=B\}.$$

\begin{defi} A pair of automorphisms $(\alpha,\beta)\in \mathrm{Aut} (A)\times \mathrm{Aut} (B)$
is said to be inducible (extensible) with respect to a non-abelian extension
 $$\mathcal{E}:0\longrightarrow B\stackrel{i}{\longrightarrow} E\stackrel{p}{\longrightarrow}A\longrightarrow0$$
if $(\alpha,\beta)$ is an image of $K$, or equivalently,
there is an automorphism $\gamma\in \mathrm{Aut}_{B}
(E)$ such that $i\beta=\gamma i,~p\gamma=\alpha p$, that is, the following commutative diagram holds:
\[\xymatrix@C=20pt@R=20pt{0\ar[r]&B\ar[d]_\beta\ar[r]^i&E\ar[d]_\gamma\ar[r]^p&A
\ar[d]_\alpha\ar[r]&0\\0\ar[r]&B\ar[r]^i&E\ar[r]^p&A\ar[r]&0.}\]
\end{defi}
It is natural to ask: when  a pair of automorphisms $(\alpha,\beta)\in \mathrm{Aut} (A)\times \mathrm{Aut} (B)$
is inducible? We discuss this problem in the following.

\begin{thm} \label{EC} Let $0\longrightarrow B\stackrel{i}{\longrightarrow}
E\stackrel{p}{\longrightarrow}A\longrightarrow0$ be a non-abelian extension of $A$ by
$B$ with a section $s$ of $p$ and
$(l_{\succ},r_{\succ},l_{\prec},r_{\prec},\omega_1,\omega_2)$ the corresponding non-abelian 2-cocycle
induced by $s$. A pair $(\alpha,\beta)\in \mathrm{Aut}(A)\times \mathrm{Aut}(B)$ is inducible if and only if there is a
linear map $\varphi:A\longrightarrow B$
satisfying the following conditions:
\begin{equation}\label{Iam1}
     \beta(l_{\succ}(x)a)-l_{\succ}(\alpha(x))\beta(a)=\varphi(x)\succ_B\beta(a),\ \ \
     \beta(r_{\succ}(x)a)-r_{\succ}(\alpha(x))\beta(a)=\beta(a)\succ_B \varphi(x),\end{equation}
\begin{equation}\label{Iam2}
     \beta(l_{\prec}(x)a)-l_{\prec}(\alpha(x))\beta(a)=\varphi(x)\prec_B \beta(a),\  \ \
     \beta(r_{\prec}(x)a)-r_{\prec}(\alpha(x))\beta(a)=\beta(a)\prec_B \varphi(x),
\end{equation}
\begin{equation}\label{Iam3}
     \beta\omega_1(x,y)-\omega_1(\alpha(x),\alpha(y))=\varphi(x)\succ_{B}\varphi(y)-\varphi(x\succ_A y)+l_{\succ}(\alpha(x))\varphi(y)+r_{\succ}(\alpha(y))\varphi(x),
\end{equation}
\begin{equation}\label{Iam4}
     \beta\omega_2(x,y)-\omega_2(\alpha(x),\alpha(y))=\varphi(x)\prec_{B}\varphi(y)-\varphi(x\prec_A y)+l_{\prec}(\alpha(x))\varphi(y)+r_{\prec}(\alpha(y))\varphi(x),
\end{equation}
for all $x,y\in A$ and $a\in B$.
\end{thm}

\begin{proof} Assume that $(\alpha,\beta)\in \mathrm{Aut}(A)\times \mathrm{Aut}(B)$ is inducible, that is, there is an
automorphism $\gamma\in \mathrm{Aut}_{B}(E)$ such
that $\gamma i=i\beta$ and $p\gamma =\alpha p$. Due to $s$ being a section of $p$,
for all
$x\in A$,
$$p(s\alpha-\gamma s)(x)=\alpha(x)-\alpha(x)=0,$$
which implies that $(s\alpha-\gamma s)(x)\in \mathrm{ker}p=B$.
So we can define a linear map $\varphi:A\longrightarrow
B$ by
\begin{equation*}\varphi(x)=(\gamma s-s\alpha)(x),~~\forall~x\in A.\end{equation*}
 Using Eqs.~(\ref{cm15}) and (\ref{cm16}), for $x,y\in A, a\in B$, we get
\begin{align*}&
    \beta\omega_2(x,y)-\omega_2(\alpha(x),\alpha(y))
      \\=&\beta(s(x)\prec_{E}s(y)-s(x\prec_{A}y))-s(\alpha(x))\prec_{E}s(\alpha(y))+s(\alpha(x)\prec_{A}\alpha(y))
      \\=&\gamma s(x)\prec_{E}\gamma s(y)-\gamma s(x\prec_{A}y))-\gamma s(\alpha(x)\prec_{E}\alpha(y))+s\alpha(x\prec_{A}y)
      \\=&\varphi(x)\prec_{E}\gamma s(y)+s\alpha(x)\prec_{E}\varphi(y)-\varphi(x\prec_{A}y)
      \\=&\varphi(x)\prec_{E}\varphi(y)+\varphi(x)\prec_{E}s\alpha(y)
      +s\alpha(x)\prec_{E}\varphi(y)-\varphi(x\prec_{A}y)
      \\=&\varphi(x)\prec_{E}\varphi(y)+r_{\prec}(\alpha(y))\varphi(x)
      +l_{\prec}(\alpha(x))\varphi(y)-\varphi(x\prec_{A}y),
\end{align*}
which indicates that Eq.~(\ref{Iam4}) holds. Take the same procedure, we can prove that Eqs.~(\ref{Iam1})-(\ref{Iam3}) hold.

Conversely, suppose that $(\alpha,\beta)\in \mathrm{Aut}(A)\times
\mathrm{Aut}(B)$ and there is a linear map $\varphi:A\longrightarrow B$ 
satisfying Eqs. (\ref{Iam1})-(\ref{Iam4}). Since $s$ is a section of $p$,
all $\hat{w}\in E$ can be written as
$\hat{w}=a+s(x)$ for some $a\in B,x\in A.$
Define a linear map $\gamma:E\longrightarrow
E$ by
\begin{equation}\label{Iam5}\gamma(\hat{w})=\gamma(a+s(x))=\beta(a)+\varphi(x)+s\alpha(x).\end{equation}
It is easy to check that $i\beta=\gamma i,~p\gamma=\alpha p$ and $\gamma(B)=B$.
In the sequel, firstly, we prove that $\gamma$ is bijective.
Indeed if $\gamma(\hat{w})=0,$ we have $s\alpha(x)=0$ and
$\beta(a)-\varphi(x)=0$. In view of $s$ and $\alpha$ being
injective, we get $x=0$, which follows that $a=0$. Thus,
$\hat{w}=a+s(x)=0$, that is $\gamma$ is injective. For any
$\hat{w}=a+s(x)\in E$,
\begin{equation*}\gamma(\beta^{-1}(a)-\beta^{-1}\varphi\alpha^{-1}(x)+s\alpha^{-1}(x))=a+s(x)=\hat{w},\end{equation*}
which means that $\gamma$ is surjective. In all, $\gamma$ is
bijective.

Secondly, we show that $\gamma$ is a homomorphism of the anti-dendriform algebra $E$. 
Indeed, for all $\hat{w}_i=a_i+s(x_i)\in E~(i=1,2)$,
\begin{align*}&
      \gamma(\hat{w}_1)\succ_{E}\gamma(\hat{w}_2)
\\=&(\beta(a_1)+\varphi(x_1)+s\alpha(x_1))\succ_{E}(\beta(a_2)+\varphi(x_2)+s\alpha(x_2))
\\=&\beta(a_1)\succ_{E}\beta(a_2)+\beta(a_1)\succ_{E}\varphi(x_2)
+\beta(a_1)\succ_{E}s\alpha(x_2)
+\varphi(x_1)\succ_{E}\beta(a_2)+\varphi(x_1)\succ_{E}\varphi(x_2)\\&+\varphi(x_1)\succ_{E}s\alpha(x_2)
+s\alpha(x_1)\succ_{E}\beta(a_2)+
s\alpha(x_1)\succ_{E}\varphi(x_2)
+s\alpha(x_1)\succ_{E}s\alpha(x_2)
\\=&\beta(a_1\succ_{B} a_2)+\beta(a_1)\succ_{B}\varphi(x_2)
+r_{\succ}(\alpha(x_2))\beta(a_1)
+\varphi(x_1)\succ_{B}\beta(a_2)+\varphi(x_1)\succ_{B}\varphi(x_2)\\&+r_{\succ}(\alpha(x_2))
\varphi(x_1)+l_{\succ}(\alpha(x_1))\beta(a_2)+l_{\succ}(\alpha(x_1))\varphi(x_2)
+\omega_1(\alpha(x_1),\alpha(x_2))+s(\alpha(x_1)\succ_A \alpha(x_2))
\end{align*}
and
\begin{align*}&
      \gamma(\hat{w}_1\succ_{E}\hat{w}_2)
      \\=&\gamma((a_1+s(x_1))\succ_{E}(a_2
+s(x_2))
\\=&\gamma(a_1\succ_{E}a_2+a_1\hat{\circ}s(x_2)+s(x_1)\succ_{E}a_2
+s(x_1)\succ_{E}s(x_2))
\\=&\gamma(a_1\succ_{E}a_2+ l_{\succ}(x_1)a_2+\mu_{\succ}(x_2)
a_1 + \omega(x_1,x_2) +s(x_1\succ_{A} x_2))
\\=&\beta(a_1\succ_{B} a_2)+ \beta(l_{\succ}(x_1)a_2+\mu_{\succ}(x_2)
a_1) +\beta\omega_1(x_1,x_2)+\varphi(x_1\succ_{A} x_2)+s\alpha(x_1\succ_{A} x_2).
\end{align*}
By Eqs.~(\ref{Iam1})-(\ref{Iam4}), we have
$$ \gamma(\hat{w}_1\succ_{E}\hat{w}_2)=\gamma(\hat{w}_1)\succ_{E}\gamma(\hat{w}_2),$$
which implies that $\gamma$ is a homomorphism of anti-dendriform algebras.
 This proof is finished.
\end{proof}

Let $(A,\succ_A,\prec_A)$ and $(B,\succ_B,\prec_B)$ be
two anti-dendriform algebras. Assume that
$\mathcal{E}:0\longrightarrow B\stackrel{i}{\longrightarrow} E\stackrel{p}{\longrightarrow} A\longrightarrow0$ is
a non-abelian extension of $(A,\succ_A,\prec_A)$ by $(B,\succ_B,\prec_B)$ with a section $s$ of $p$ and
 $ (l_{\succ},r_{\succ},l_{\prec},r_{\prec},\omega_{1},\omega_{2})$ be the 
 corresponding non-abelian 2-cocycle induced by $s$.

 For all $(\alpha,\beta)\in \mathrm{Aut}(A
)\times \mathrm{Aut}(B)$, define a bilinear map $\omega_{(1,\alpha,\beta)},\omega_{(2,\alpha,\beta)}:A\times
A\rightarrow B$ and linear maps 
$l_{(\succ,\alpha,\beta)},r_{(\succ,\alpha,\beta)},l_{(\prec,\alpha,\beta)},r_{(\prec,\alpha,\beta)}:A
\rightarrow End(B)$
 respectively by
 \begin{equation}\label{Inc1}l_{(\succ,\alpha,\beta)}(x)a=\beta l_{\succ}(\alpha^{-1}(x))\beta^{-1}(a), \ \ \ 
 r_{(\succ,\alpha,\beta)}(x)a=\beta r_{\succ}(\alpha^{-1}(x))\beta^{-1}(a),\end{equation}
 \begin{equation}\label{Inc2}l_{(\prec,\alpha,\beta)}(x)a=\beta l_{\prec}(\alpha^{-1}(x))\beta^{-1}(a), \ \ \
  r_{(\prec,\alpha,\beta)}(x)a=\beta r_{\prec}(\alpha^{-1}(x))\beta^{-1}(a),\end{equation}
  \begin{equation}\label{Inc3}\omega_{(1,\alpha,\beta)}(x,y)=\beta\omega_1(\alpha^{-1}(x),\alpha^{-1}(y)), \ \ \ \omega_{(2,\alpha,\beta)}(x,y)=\beta\omega_2(\alpha^{-1}(x),\alpha^{-1}(y))\end{equation}
for all $x,y\in A,a\in V.$

\begin{pro} With the above notations,
$(l_{(\succ,\alpha,\beta)},r_{(\succ,\alpha,\beta)},l_{(\prec,\alpha,\beta)},r_{(\prec,\alpha,\beta)},
\omega_{(1,\alpha,\beta)},\omega_{(2,\alpha,\beta)})$
is a non-abelian 2-cocycle on $(A,\succ_A,\prec_A)$ with values in $(B,\succ_B,\prec_B)$. \end{pro}

\begin{proof} Instead of
proving the Eqs.\eqref{C1}-\eqref{C11} hold, which requires a long
and laborious computation. We give a simple proof. Since $(l_{\succ},r_{\succ},l_{\prec},r_{\prec},\omega_{1},\omega_{2})$ 
is a non-abelian 2-cocycle of $(A,\succ_A,\prec_A)$ with values in $(B,\succ_B,\prec_B)$, $(A\sharp B,\succ,\prec)$ 
is an anti-dendriform
algebra with $(\succ,\prec)$ defined as follows:
\begin{align}\label{Inc4}
 &(x,a)\succ (y,b)=(x\succ_A y,a\succ_V b+l_{\succ}(x)b+r_{\succ}(y)a+\omega_1(x, y)),\\
    \label{Inc5} & (x,a)\prec (y,b)=(x\prec_A y,a\prec_V b+l_{\prec}(x)b+r_{\prec}(y)a+\omega_2(x, y)),\ \ \forall~~x,y \in A,a,b \in B.
\end{align}
Define a linear isomorphism
 \begin{equation}\label{Inc6}\varphi: A\sharp_{(\alpha,\beta)} B \longrightarrow A\sharp B,
 \ \ \ \varphi(x,a)=(\alpha^{-1}(x),\beta(a)),~\forall~a\in B,~x\in A\end{equation}  
  with the inverse  
  \begin{equation}\label{Inc7}\varphi^{-1}:A\sharp B \longrightarrow A\sharp_{(\alpha,\beta)} B,\ \ \ 
  \varphi^{-1}(x,a)=(\alpha(x),\beta^{-1}(a)),~\forall~a\in B,~x\in A.\end{equation} 
In order to verify that $(l_{(\succ,\alpha,\beta)},r_{(\succ,\alpha,\beta)},l_{(\prec,\alpha,\beta)},r_{(\prec,\alpha,\beta)},
\omega_{(1,\alpha,\beta)},\omega_{(2,\alpha,\beta)})$
is a non-abelian 2-cocycle on $(A,\succ_A,\prec_A)$ with values in $(B,\succ_B,\prec_B)$, 
To verify the conclusion,
we only need to prove that the anti-dendriform algebra
structure on $(\succ_{\alpha,\beta},\prec_{\alpha,\beta})$ on $A\oplus B$, 
denoted the anti-dendriform algebra simply by $A\sharp_{(\alpha,\beta)} B$, such that $\varphi$ is an isomorphism
of anti-dendriform algebras is given by:
 \begin{align*}&(x,a)\succ_{(\alpha,\beta)} (y,b)
=(x\succ_A y,a\succ_V b+l_{(\succ,\alpha,\beta)}(x)b
+r_{(\succ,\alpha,\beta)}(y)a+\omega_{(1,\alpha,\beta)}(x,y)),\\
&(x,a)\prec_{(\alpha,\beta)} (y,b)
=(x\prec_A y,a\prec_V b+l_{(\prec,\alpha,\beta)}(x)b
+r_{(\prec,\alpha,\beta)}(y)a+\omega_{(2,\alpha,\beta)}(x,y))
\end{align*}
 for all $a,b\in B$ and $x,y\in A$.
Indeed, by Eqs. \eqref{Inc1}-\eqref{Inc7}, we have 
\begin{align*}&(x,a)\succ_{(\alpha,\beta)} (y,b)=\varphi^{-1} \varphi((x,a)\succ'(y,b))
\\=&\varphi^{-1}( \varphi(x,a)\succ \varphi(y,b))
\\=&\varphi^{-1}( (\alpha^{-1}(x),\beta(a))\succ (\alpha^{-1}(y),\beta(b)))
\\=&\varphi^{-1}(\alpha^{-1}(x)\succ_A \alpha^{-1}(y),\beta(a)\succ_V \beta(b)+l_{\succ}(\alpha^{-1}(x))\beta(b)
+r_{\succ}(\alpha^{-1}(y))\beta(a)+\omega_1(\alpha^{-1}(x), \alpha^{-1}(y)))
\\=&(x\succ_A y,a\succ_V b+\beta^{-1}(l_{\succ}(\alpha^{-1}(x))\beta(b))
+\beta^{-1}(r_{\succ}(\alpha^{-1}(y))\beta(a))+\beta^{-1}\omega_1(\alpha^{-1}(x), \alpha^{-1}(y)))
\\=&(x\succ_A y,a\succ_V b+l_{(\succ,\alpha,\beta)}(x)b
+r_{(\succ,\alpha,\beta)}(y)a+\omega_{(1,\alpha,\beta)}(x,y)),
\end{align*}
for all $a,b\in B$ and $x,y\in A$. 
By the same token, 
\begin{align*}(x,a)\prec_{(\alpha,\beta)} (y,b)
=(x\prec_A y,a\prec_V b+l_{(\prec,\alpha,\beta)}(x)b
+r_{(\prec,\alpha,\beta)}(y)a+\omega_{(2,\alpha,\beta)}(x,y)).
\end{align*}
 The proof is finished.
\end{proof}

\begin{thm}\label{CM2}Let $0\longrightarrow B\stackrel{i}{\longrightarrow}
E\stackrel{p}{\longrightarrow} A
\longrightarrow0$ be a non-abelian extension of an anti-dendriform algebra
 $(A,\succ_A,\prec_A)$ by $(B,\succ_B,\prec_B)$ with a section $s$ of $p$ and 
 $(l_{\succ},r_{\succ},l_{\prec},r_{\prec},\omega_1,\omega_2)$ be the
corresponding non-abelian 2-cocycle induced by $s$. A pair $(\alpha,\beta)\in \mathrm{Aut}(A
)\times \mathrm{Aut}(B)$ is inducible if and only if 
 $(l_{\succ},r_{\succ},l_{\prec},r_{\prec},\omega_1,\omega_2)$ and 
$(l_{(\succ,\alpha,\beta)},r_{(\succ,\alpha,\beta)},l_{(\prec,\alpha,\beta)},r_{(\prec,\alpha,\beta)},
\omega_{(1,\alpha,\beta)},\omega_{(2,\alpha,\beta)})$
are equivalent non-abelian 2-cocycles.\end{thm}

\begin{proof}
Assume that $(\alpha,\beta)\in \mathrm{Aut}(A)\times \mathrm{Aut}(B)$
is inducible, then by Theorem \ref{EC}, there is a linear map
$\varphi:A\rightarrow B$ satisfying
Eqs.~(\ref{Iam1})-(\ref{Iam4}). For all $x,y\in A$, there are $x_0,y_0\in A$ such that $x=\alpha(x_0),~y=\alpha(y_0)$. Thus,
by Eqs.~(\ref{Iam3}) and (\ref{Inc3}), we obtain
\begin{align*}&
\omega_{(1,\alpha,\beta)}(x,y)-\omega_{1}(x,y)
\\=& \beta\omega_1(\alpha^{-1}(x),\alpha^{-1}(y))-\omega_{1}(x,y)
\\=&\beta\omega_1(x_0,y_0))-\omega_{1}(\alpha(x_0),\alpha(y_0))
\\=&\varphi(x_0)\succ_{B}\varphi(y_0)-\varphi(x_0\succ_A y_0)+l_{\succ}(\alpha(x_0))\varphi(y_0)+r_{\succ}(\alpha(y_0))\varphi(x_0).
 \end{align*}
By the same token, we obtain
 \begin{equation*}\omega_{(2,\alpha,\beta)}(x,y)-\omega_{2}(x,y)
=\varphi(x_0)\succ_{B}\varphi(y_0)-\varphi(x_0\succ_A y_0)+l_{\prec}(\alpha(x_0))\varphi(y_0)+r_{\prec}(\alpha(y_0))\varphi(x_0),
 \end{equation*}
   \begin{equation*}
     \beta(l_{\succ}(x_0)a)-l_{\succ}(\alpha(x_0))\beta(a)=\varphi(x_0)\succ_A \beta(a),\ \ \
     \beta(r_{\succ}(x_0)a)-r_{\succ}(\alpha(x_0))\beta(a)=\beta(a)\succ_A \varphi(x_0),\end{equation*}
\begin{equation*}
     \beta(l_{\prec}(x_0)a)-l_{\prec}(\alpha(x_0))\beta(a)=\varphi(x_0)\prec_A \beta(a),\  \ \
     \beta(r_{\prec}(x_0)a)-r_{\prec}(\alpha(x_0))\beta(a)=\beta(a)\prec_A \varphi(x_0).
\end{equation*}
  Thus, $(l_{\succ},r_{\succ},l_{\prec},r_{\prec},\omega_1,\omega_2)$ and 
$(l_{(\succ,\alpha,\beta)},r_{(\succ,\alpha,\beta)},l_{(\prec,\alpha,\beta)},r_{(\prec,\alpha,\beta)},
\omega_{(1,\alpha,\beta)},\omega_{(2,\alpha,\beta)})$ are cohomologous via the linear map
$\varphi\alpha^{-1}:A\rightarrow B$.

The converse part can be checked analogously.
\end{proof}

%%%%%%%%%%%%%%%%%%%%%%%%%%%%%%%%%%%%%%%%%%%%%%%%%%%%%%%%%%%%%%%%%%%%%%%%%%%%%%%%%%%%%%%%%%%%%%%%%%%%%%%%%%
\subsection{Wells exact sequences for anti-dendriform algebras}
In this section, we consider the Wells map associated with non-abelian extensions of anti-dendriform algebras.
Then we interpret the results gained in Section 4.1 in terms of the Wells map.

Let $\mathcal{E}:0\longrightarrow B\stackrel{i}{\longrightarrow}
E\stackrel{p}{\longrightarrow} A\longrightarrow0$
be a non-abelian extension of $(A,\succ_A,\prec_A)$ by $(B,\succ_B,\prec_B)$ together with a section $s$ of $p$.
Assume that $(l_{\succ},r_{\succ},l_{\prec},r_{\prec},\omega_1,\omega_2)$ is the corresponding non-abelian 2-cocycle induced by $s$.
Define a linear map $W:\mathrm{Aut}(A)\times \mathrm{Aut}(B)\longrightarrow H^2_{nab}(A,B)$ by
\begin{equation}\label{W1}
W(\alpha,\beta)=[(l_{(\succ,\alpha,\beta)},r_{(\succ,\alpha,\beta)},l_{(\prec,\alpha,\beta)},r_{(\prec,\alpha,\beta)},
\omega_{(1,\alpha,\beta)},\omega_{(2,\alpha,\beta)})-(l_{\succ},r_{\succ},l_{\prec},r_{\prec},\omega_1,\omega_2)].
\end{equation}
The map $W$ is called the Wells map. Based on Lemma \ref{CM}, one can easily check that
the Wells map $W$ does not depend on the choice of sections.

\begin{thm} \label{Wm3} Let
$\mathcal{E}:0\longrightarrow B\stackrel{i}{\longrightarrow} E\stackrel{p}{\longrightarrow}A\longrightarrow0$
be a non-abelian extension of
$(A,\succ_A,\prec_A)$ by $(B,\succ_B,\prec_B)$ together with a section $s$ of $p$.
There is an exact sequence:
$$1\longrightarrow \mathrm{Aut}_{B}^{A}(E)\stackrel{T}{\longrightarrow} \mathrm{Aut}_{B}(E)\stackrel{K}{\longrightarrow}\mathrm{Aut}(A)\times \mathrm{Aut}(B)\stackrel{W}{\longrightarrow} H^2_{nab}(A,B),$$
where $\mathrm{Aut}_{B}^{A}(E)=\{\gamma \in \mathrm{Aut}(E)| K(\gamma)=(I_{A},I_{B}) \}$.\end{thm}

\begin{proof} It follows from Theorem \ref{CM2}. 
\end{proof}
Let
$\mathcal{E}:0\longrightarrow B\stackrel{i}{\longrightarrow} E
\stackrel{p}{\longrightarrow}A\longrightarrow0$
be a non-abelian extension of
$(A,\succ_A,\prec_A)$ by $(B,\succ_B,\prec_B)$ together with a section $s$ of $p$. Assume that
$(l_{\succ},r_{\succ},l_{\prec},r_{\prec},\omega_1,\omega_2)$ is a non-abelian 2-cocycle induced by $s$. Denote
\begin{align}\label{W5} Z_{nab}^{1}(A,B)=&\{\varphi:A\rightarrow B|\varphi(x\succ_A y)-\varphi(x)\succ_B\varphi(y)
=l_{\succ}(x)\varphi(y)+r_{\succ}(y)\varphi(x),
\\ \notag&\varphi(x\prec_A y)-\varphi(x)\prec_B\varphi(y)=l_{\prec}(x)\varphi(y)+r_{\prec}(y)\varphi(x),
\\ \notag&\varphi(x)\succ_B a=a\succ_B \varphi(x)=\varphi(x)\prec_B a=a\prec_B \varphi(x)=0,~~\forall~x,y\in A,a\in B\}.\end{align}
It is easy to check that $Z_{nab}^{1}(A,B)$ is an abelian group. The abelian group $Z_{nab}^{1}(A,B)$ is called a non-abelian 1-cocycle.

\begin{pro}\label{Wm4} With the above notations, we have

  (i) The linear map $S:\mathrm{Ker} K\longrightarrow Z_{nab}^{1}(A,B)$ given by
  \begin{equation}\label{W6}S(\gamma)(x)=\varphi_{\gamma}(x)=\gamma s(x)-s(x),~\forall~~\gamma\in \mathrm{Ker} K,~x \in A\end{equation} is
a homomorphism of groups.

(ii) $S$ is an isomorphism, that is, $\mathrm{Ker }K\simeq Z_{nab}^{1}(A,B)$ as groups.\end{pro}
\begin{proof} (i) By Eqs.~(\ref{cm14})-(\ref{cm16}) and (\ref{W6}), for all $x,y\in A$, we have,
\begin{align*}&\varphi_{\gamma}(x\succ_A y)-\varphi_{\gamma}(x)\succ_B\varphi_{\gamma}(y)
-l_{\succ}(x)\varphi_{\gamma}(y)-r_{\succ}(y)\varphi_{\gamma}(x)\\=&
\gamma s(x\succ_A y)-s(x\succ_A y)-(\gamma s(x)-s(x))\succ_E (\gamma s(y)-s(y))
-s(x)\succ_E\varphi_{\gamma}(y)-\varphi_{\gamma}(x)\succ_{E}s(y)
\\=&
\gamma s(x\succ_A y)-s(x\succ_A y)-\gamma s(x)\succ_{E}\gamma s(y)+ s(x)\succ_{E} s(y)
\\=&
\omega(x, y)-\gamma \omega(x, y)
\\=&0.
\end{align*}
Analogously, $\varphi_{\gamma}$ satisfies other equations in \eqref{W5}.
Therefore, $\varphi_{\gamma}\in Z_{nab}^{1}(A,B)$
and $S$ is well defined.
For any $\gamma_1,\gamma_2\in \mathrm{Ker} K$ and $x\in A$, suppose $S(\gamma_1)=\varphi_{\gamma_1},S(\gamma_2)=\varphi_{\gamma_2}$,
by Eq.~(\ref{W6}), we get
\begin{align*}S(\gamma_1 \gamma_2)(x)&=\gamma_1 \gamma_2s(x)-s(x)
\\&=\gamma_1(\varphi_{\gamma_2}(x)+s(x))-s(x)
\\&=\gamma_1\varphi_{\gamma_2}(x)+\varphi_{\gamma_1}(x)
\\&=\varphi_{\gamma_2}(x)+\varphi_{\gamma_1}(x)
\\&=S(\gamma_2)(x)+S(\gamma_1)(x),\end{align*}
which means that $S$ is a homomorphism of groups.

(ii) For all $\gamma\in \mathrm{Ker}K$, then $K(\gamma)=(p\gamma s,\gamma|_{B})=(I_A,I_B)$. If $S(\gamma)=\varphi_{\gamma}=0$, we get that
 $\varphi_{\gamma}(x)=\gamma s(x)-s(x)=0$, thus $\gamma=I_{E}$, which indicates that $S$ is injective.
Secondly, we check that $S$ is surjective.
 For any $\varphi\in Z_{nab}^{1}(A,B)$, define a linear map $\gamma:E\rightarrow E$ by
  \begin{equation}\label{W7}\gamma(\hat{x})=\gamma(a+s(x))=a+\varphi(x)+s(x),~\forall~\hat{x}\in E.\end{equation}
It is obviously that $(p\gamma s,\gamma|_{B})=(I_A,I_B)$ and $\gamma$ is bijective.
Furthermore, we should prove that $\gamma$ is a homomorphism of the anti-dendriform algebra $E$. Indeed,
 we can take the same procedure of the proof of the converse part of Theorem \ref{EC}.
 Thus, $\gamma\in \mathrm{Ker} K$. It follows that
$S$ is surjective.
In all, $S$ is bijective.
 So $\mathrm{Ker }K\simeq Z_{nab}^{1}(A,B)$.
\end{proof}

Combining Theorem \ref{Wm3} and Proposition \ref{Wm4}, we have

\begin{thm}\label{Wm5} Let $\mathcal{E}:0\longrightarrow B\stackrel{i}{\longrightarrow} E
\stackrel{p}{\longrightarrow}A\longrightarrow0$
be a non-abelian extension of
$(A,\succ_A,\prec_A)$ by $(B,\succ_B,\prec_B)$ together with a section $s$ of $p$. Then there is an exact sequence:
$$0\longrightarrow Z_{nab}^{1}(A,B)\stackrel{i}{\longrightarrow} \mathrm{Aut}_{B}(E)\stackrel{K}{\longrightarrow}\mathrm{Aut}(A)\times \mathrm{Aut}(B)\stackrel{W}{\longrightarrow} H^2_{nab}(A,B).$$
\end{thm}

%%%%%%%%%%%%%%%%%%%%%%%%%%%%%%%%%%%%%%%%%%%%%%%%%%%%%%%%%%%%%%%%%%%%%%%%%%%%%%%%%%%%%%%%%%%%%%%%%%%%%%%%%%
\section{Bicrossed product and matched pairs }

Let 
$\Omega(A,V)=(l_{\succ},r_{\succ},l_{\prec},r_{\prec},\rho_{\succ},\mu_{\succ},\rho_{\prec},\mu_{\prec},\varpi_1,\varpi_2,\succ_V,\prec_V)$ be
  be an extending datum of $A$ through $V$. If $\varpi_1$ and $\varpi_2$ are trivial maps, by Proposition \ref{U1}, 
  we get that $\Omega(A,V)$
 is an extending structure of $A$ through $V$ if and only if 
 $(V,\succ_V,\prec_V)$ is an anti-dendriform algebra,
$(A, l_{\succ},r_{\succ},l_{\prec},r_{\prec}) $ is a bimodule of $V$, 
$(V,\rho_{\succ},\mu_{\succ},\rho_{\prec},\mu_{\prec})$ is a bimodule of $A$
and they satisfy Eqs. (\ref{S2})-(\ref{S4}), (\ref{S6}),(\ref{S8}), (\ref{S10}) and (\ref{S13})-(\ref{S15}). 
In this case, the associated
unified product
$A\natural V$ is denoted by $A \bowtie V$, called the bicrossed product of $A$ by $V$ and $(A,V,l_{\succ},r_{\succ},l_{\prec},r_{\prec},\rho_{\succ},\mu_{\succ},\rho_{\prec},\mu_{\prec})$ is
called the matched pair of anti-dendriform algebras $A$ and $V$. 
For the completeness, we write them in detail.

\begin{pro} \label{M0} Let $(A_{1},\succ_{1},\prec_{1})$ and
$(A_{2},\succ_{2},\prec_{2})$ be two anti-dendriform algebras. Suppose that there are linear maps
$l_{\succ_1},r_{\succ_1},l_{\prec_1},r_{\prec_1}:A_1\longrightarrow \hbox{End}(A_2)$
and $l_{\succ_2},r_{\succ_2},l_{\prec_2},r_{\prec_2}:A_2\longrightarrow \hbox{End}(A_1)$. 
Define multiplications $(\succ,\prec)$ on $A_1\oplus A_2$ by 
\begin{equation*}(x,a)\succ(y,b)=(x\succ_{1}y+l_{\succ_2}(a)y+r_{\succ_2}(b)x,a\succ_{2}b+l_{\succ_1}(x)b+r_{\succ_1}(y)a)
,\end{equation*}
\begin{equation*}(x,a)\prec(y,b)=(x\prec_{1}y+l_{\prec_2}(a)y+r_{\prec_2}(b)x,a\prec_{2}b+l_{\prec_1}(x)b+r_{\prec_1}(y)a),~~\forall~x,y\in A_1,a,b\in A_2.
\end{equation*}
Then $(A_{1}\oplus A_{2},\succ,\prec)$
is an anti-dendriform algebra if and only if the following conditions hold:
\begin{enumerate}
	\item $(A_2,l_{\succ_1},r_{\succ_1},l_{\prec_1},r_{\prec_1})$ is a
	representation of $(A_1,\succ_1,\prec_1)$.
	\item $(A_1,l_{\succ_2},r_{\succ_2},l_{\prec_2},r_{\prec_2})$ is a
	representation of $(A_2,\succ_2,\prec_2)$.
	\item The
	following compatible conditions hold:
\begin{align}\label{M1}
&x\succ_{1}(r_{\succ_2}(c)y)+r_{\succ_2}(l_{\succ_1}(y)c)x=-r_{\succ_2}(c)(x\cdot y)\\=&-x\prec_{1}(r_{\cdot_2}(c)y)
-r_{\prec_2}(l_{\cdot_{1}}(y)c)x=r_{\prec_2}(c)(x\prec_{1} y),\nonumber\end{align}
\begin{align}\label{M2}&
x\succ_{1}(l_{\succ_2}(b)z)+r_{\succ_2}(r_{\succ_1}(z)b)x=-(r_{\cdot_2}(b)x)\succ_{1}z
-l_{\succ_2}(l_{\cdot_1}(x)b)z
\\=&-x\prec_{1}(l_{\cdot_2}(b)z)
-r_{\prec_2}(r_{\cdot_1}(z)b)x=(r_{\prec_2}(b)x)\prec_{1} z+l_{\prec_2}(l_{\prec_1}(x)b)z,\nonumber\end{align}
\begin{align}\label{M3}
&l_{\succ_2}(a)(y\succ_{1}z)=-(l_{\cdot_2}(a)y)\succ_{1}z-l_{\succ_2}(r_{\cdot_1}(y)a)z
=-l_{\prec_{2}}(a)(y\cdot_{1}z)
\\\notag=&(l_{\prec_2}(a)y)\prec_{1}z+l_{\prec_2}(r_{\prec_1}(y)a) z,\end{align}
\begin{align}\label{M4}
&l_{\succ_1}(x)(b\succ_{2}c)=-(l_{\cdot_1}(x)b)\succ_{2}c-l_{\succ_1}(r_{\cdot_2}(b)x)c
=-l_{\prec_{1}}(x)(b\cdot_{2}c)\\=&(l_{\prec_1}(x)b)\prec_{2}c+l_{\prec_1}(r_{\prec_2}(b)x)c,\notag\end{align}
\begin{align}\label{M5}&
a\succ_{2}(l_{\succ_1}(y)c)+r_{\succ_1}(r_{\succ_2}(c)y)a=-(r_{\cdot_1}(y)a)\succ_{2}c
-l_{\succ_1}(l_{\cdot_2}(a)y)c\\=&-x\prec_{2}(l_{\cdot_1}(y)c)
-r_{\prec_1}(r_{\cdot_2}(c)y)a=(r_{\prec_1}(y)a)\prec_{2} c+l_{\prec_1}(l_{\prec_2}(a)y)c,\nonumber\end{align}
\begin{align}\label{M6}&a\succ_{2}(r_{\succ_1}(z)b)+r_{\succ_1}(l_{\succ_2}(b)z)a=-r_{\succ_1}(z)(a\cdot b)\notag
\\=&-a\prec_{2}(r_{\cdot_1}(z)b)
-r_{\prec_1}(l_{\cdot_{2}}(b)z)a=r_{\prec_1}(z)(a\prec_{2} b),\nonumber\end{align}
\begin{equation}\label{M7}
r_{\prec_2}(c)(x\succ_{1}y)=x\succ_{1}(r_{\prec_2}(c)y)+r_{\succ_2}(l_{\prec_1}(y)c)x,\end{equation}
\begin{equation}\label{M8}
(r_{\succ_2}(b)x)\prec_{1}z+l_{\prec_2}(l_{\succ_1}(x)b)z=x\succ_{1}(l_{\prec_2}(b)z)+
r_{\succ_2}(r_{\prec_1}(z)b)x,\end{equation}
\begin{equation}\label{M9}
(l_{\succ_2}(a)y)\prec_{1}z+l_{\prec_2}(r_{\succ_1}(y)a)z=
l_{\succ_2}(a)(y\prec_1z),\end{equation}
\begin{equation}\label{M10}
(l_{\succ_1}(x)b)\prec_{2}c+l_{\prec_1}(r_{\succ_2}(b)x)c=
l_{\succ_1}(x)(b\prec_2c),\end{equation}
\begin{equation}\label{M11}
(r_{\succ_1}(y)a)\prec_{2}c+l_{\prec_1}(l_{\succ_2}(a)y)c=a\succ_{2}(l_{\prec_1}(y)c)+
r_{\succ_1}(r_{\prec_2}(c)y)a,\end{equation}
\begin{equation}\label{M12}r_{\prec_1}(z)(a\succ_{2}b)=a\succ_{2}(r_{\prec_1}(z)b)+r_{\succ_1}(l_{\prec_2}(b)z)a,\end{equation}
for any $x,y\in A_{1}$ and $a,b\in A_{2}$, where
\begin{equation*}l_{\cdot_1}=l_{\succ_1}+l_{\prec_1},\ l_{\cdot_2}=l_{\succ_2}+l_{\prec_2},\ 
r_{\cdot_1}=r_{\succ_1}+r_{\prec_1},\ r_{\cdot_2}=r_{\succ_2}+r_{\prec_2}.\end{equation*}
\end{enumerate}
Denote this anti-dendriform algebra by $A_1\bowtie A_2$, and 
$(A_{1},A_{2},l_{\succ_1},r_{\succ_1},l_{\prec_1},r_{\prec_1},l_{\succ_2},r_{\succ_2},l_{\prec_2},r_{\prec_2})$
satisfying the above conditions is called a {\bf matched pair of
anti-dendriform algebras}.
\end{pro}

\begin{proof} It is a special case of Proposition \ref{U1}.
\end{proof}

\begin{cor}\label{Ma} If
$(A_{1},A_{2},l_{\succ_1},r_{\succ_1},l_{\prec_1},r_{\prec_1},l_{\succ_2},r_{\succ_2},l_{\prec_2},r_{\prec_2})$ is
a matched pair of
anti-dendriform algebras, then 
$(A_{1},A_{2},l_{\succ_1}+l_{\prec_1},r_{\succ_1}+r_{\prec_1},l_{\succ_2}+l_{\prec_2},r_{\succ_2}+r_{\prec_2})$
is a matched pair of associative algebras.
\end{cor}

\begin{proof} It follows from Remark \ref{U3}.
\end{proof}

{\bf Factorization problem}: Let $A$ and $B$ be two anti-dendriform algebras. Characterize and 
classify all anti-dendriform algebras $C$ that factorize through $A$ and $B$, 
i.e. $C$ contains $A$ and $B$ as anti-dendriform subalgebras such that $C=A+B$ and $A \cap B=\{0\}$.

On the basis of Theorem \ref{U2}, we easily get the following result.
\begin{cor}An anti-dendriform algebra $C$ factorizes through two anti-dendriform algebras $A$ and $B$ 
if and only if there is a matched pair 
$(A_{1},A_{2},l_{\succ_1},r_{\succ_1},l_{\prec_1},r_{\prec_1},l_{\succ_2},r_{\succ_2},l_{\prec_2},r_{\prec_2})$
of anti-dendriform algebras such that $C \cong A \bowtie B$.
\end{cor}
 
%%%%%%%%%%%%%%%%%%%%%%%%%%%%%%%%%%%%%%%%%%%%%%%%%%%%%%%%%%%%%%%%%%%%%%%%%%%%%%%%%%%%%%%%%%%%%%%%%%%%%
\section{Anti-dendriform D-bialgebras and Yang-Baxter equations }

In this section, we introduce the notion of anti-dendriform D-bialgebras as the bialgebra structures corresponding 
   to double construction of associative algebras with respect to the commutative Cone cocycles. Both of them
are interpreted in terms of certain matched pairs of associative algebras as well 
as the compatible anti-dendriform algebras. The study of coboundary case leads to the introduction of the AD-YBE, whose
skew-symmetric solutions give coboundary anti-dendriform D-bialgebras. 
The notion of $\mathcal O$-operators of
anti-dendriform algebras is introduced to construct skew-symmetric solutions
of the AD-YBE. We also characterize the relationship between the skew-symmetric solutions of AD-YBE and $\mathcal O$-operators.

\subsection{Anti-dendriform D-bialgebras }

 Let $(A, \cdot)$ be an associative algebra and $\omega$ be a bilinear form on $(A, \cdot)$. If $\omega$ is
symmetric and satisfies 
\begin{equation}\omega(x \cdot y, z) + \omega(y \cdot z, x)+ \omega(z \cdot x, y)=0, \forall~x, y, z \in A,\end{equation}
 then $\omega$ is called a commutative Connes cocycle.

\begin{defi} Let $(A_{1},\cdot)$ be an associative algebra. Suppose that there is an associative algebra structure $\ast$ on the dual space
$A_{1}^{*}$. We call $(A,\cdot,\omega)$ a double construction of the commutative Connes cocycle if it satisfies the conditions:
\begin{enumerate}
\item $ A=A_1\oplus A_{1}^{*}$ as direct sum of vector spaces.
\item $A$ is an associative algebra and $A_1, A_{1}^{*}$ are subalgebras of $A$.
\item $\omega$ is the natural symmetric bilinear form on $A_1\oplus A_{1}^{*}$ given by 
\begin{equation}\label{Co1}\omega(x+a, y+b) =\langle x,b\rangle+\langle a,y\rangle, \forall~x, y\in A_1,~a, b\in A_{1}^{*}\end{equation}
and $\omega$ is a commutative Connes cocycle on $A$.
\end{enumerate}
Denote it by $(A_1\oplus A_{1}^{*},A_1, A_{1}^{*},\omega)$. 
\end{defi}

\begin{pro} \cite{15}
Let $(A,\cdot)$ be an associative algebra and $\omega$ be a nondegenerate commutative
Connes cocycle on $(A,\cdot)$. Then there exists a compatible anti-dendriform algebra structure $(A,\succ,\prec)$
 on $(A,\cdot)$ defined by
\begin{equation} \label{C2}\omega (x \prec y, z)=-\omega(y, z\cdot x), \ \  \ \omega(x \succ y, z)=-\omega(x, y\cdot z), ~\forall~x, y, z \in A.\end{equation}
\end{pro}

\begin{pro}
 Let $(A\oplus A^{*},A, A^{*},\omega)$ be a double construction of the commutative Connes cocycle.
  Then there exists a compatible anti-dendriform algebra structure $(\succ,\prec)$ on $A\oplus A^{*}$ given by Eq. \eqref{C2}.
Furthermore, $(A,\succ_A,\prec_A)$ and $(A^{*},\succ_{A^{*}},\prec_{A^{*}})$ 
are anti-dendriform subalgebras whose associated associative algebras are $(A,\cdot)$ and $(A^{*},\ast)$ respectively, where
$\succ_A=\succ|_{A\otimes A},\prec_A=\prec|_{A\otimes A}$ and
 $\succ_{A^{*}}=\succ|_{A^{*}\otimes A^{*}},\prec_{A^{*}}=\prec|_{A^{*}\otimes A^{*}}$.
\end{pro}

\begin{proof}
The first part is given in the above Proposition. For all $x,y\in A$, since $x\succ y\in A\oplus A^{*}$, suppose that
$x\succ y=u+u^{*}$ with $u\in A,u^{*}\in A^{*}$. Then 
\begin{equation*} \langle u^{*}, z\rangle=\omega(x\succ_{A}y, z)=\omega(x,y\cdot z)=0
,~\forall~ z \in A,\end{equation*}
which indicates that $u^{*}=0$, that is, $x\succ_{A}y\in A$. Analogously,
$x\prec_{A}y\in A$. Thus, $(A,\succ_A,\prec_A)$ is a subalgebra of $A\oplus A^{*}$.
By the same token, $(A^{*},\succ_{A^{*}},\prec_{A^{*}})$ is a subalgebra of $A\oplus A^{*}$.
It is straightforward to prove that the associated associative algebras of 
$(A,\succ_A,\prec_A)$ and $(A^{*},\succ_{A^{*}},\prec_{A^{*}})$ 
are $(A,\cdot)$ and $(A^{*},\ast)$ respectively.
\end{proof}

Recall the matched pairs of associative algebras studied in \cite{10}. 
Let $(A_{1},\circ_{1})$ and
$(A_{2},\circ_{2})$ be two associative algebras. Suppose that there exist four linear maps
$l_{1},r_{1}:A_1\longrightarrow \hbox{End}(A_2)$
and $l_{2},r_{2}:A_2\longrightarrow \hbox{End}(A_1)$ such that
$(A_{2},l_{1},r_{1})$ and
 $(A_{1},l_{2},r_{2})$ are
 representations of the anti-dendriform algebras $(A_{1},\circ_{1})$
and $(A_{2},\circ_{2})$ respectively, and they satisfying
the following conditions:
\begin{align}\label{AM1}l_{1}(x)(a\circ_{2}b)=l_{1}(r_{2}(a)x)b+(l_{1}(x)a)\circ_2b,\end{align}
\begin{align}\label{AM2}r_{1}(x)(a\circ_{2}b)=r_{1}(l_{2}(b)x)a+a\circ_2(r_{1}(x)b),\end{align}
\begin{align}\label{AM3}l_{2}(a)(x\circ_{1}y)=l_{2}(r_{1}(x)a)y+(l_{2}(a)x)\circ_1y,\end{align}
  \begin{align}\label{AM4}r_{2}(a)(x\circ_{1}y)=r_{2}(l_{1}(y)a)x+x\circ_1(r_{2}(a)y),\end{align}
  \begin{align}\label{AM5}l_{1}(l_2(a)x)b+(r_1(x)a)\circ_{2}b-r_{1}(r_{2}(b)x)a-a\circ_2(l_{1}(x)b)=0,\end{align}
   \begin{align}\label{AM6}l_{2}(l_1(x)a)y+(r_2(a)x)\circ_{1}y-r_{2}(r_{1}(y)a)x-x\circ_1(l_{2}(b)x)=0,\end{align} 
   for all $x,y\in A_1$ and $a,b\in A_2$. Then 
 $(A_{1},A_{2},l_{1},r_{1},l_{2},r_{2})$ is called a matched pair of associative algebras.
 
\begin{thm}
 Let $(A,\succ_{A},\prec_{A})$ and $(A^{*},\succ_{A^{*}},\prec_{A^{*}})$ be two anti-dendriform algebras and their associated
 associative algebras be $(A,\cdot)$ and $(A^{*},\ast)$ respectively. Then the following conditions are equivalent:
\begin{enumerate}
\item There is a double construction $(A\oplus A^{*},\omega,A, A^{*})$ of the commutative Connes cocycle
 such that the compatible anti-dendriform algebra $(A\oplus A^{*},\succ,\prec)$ defined by Eq. \eqref{C2} contains
 $(A,\succ_{A},\prec_{A})$ and $(A^{*},\succ_{A^{*}},\prec_{A^{*}})$ as anti-dendriform subalgebras.
 \item $(A,A^{*},-(R_{\prec_A}^{*}+R_{\succ_A}^{*}),L_{\prec_A}^{*},R_{\succ_A}^{*},-(L_{\prec_A}^{*}+L_{\succ_A}^{*}),
 -(R_{\prec_{A^{*}}}^{*}+R_{\succ_{A^{*}}}^{*}),L_{\prec_{A^{*}}}^{*},R_{\succ_{A^{*}}}^{*},-(L_{\prec_{A^{*}}}^{*}+L_{\succ_{A^{*}}}^{*})$ 
 is a matched pair of
anti-dendriform algebras.
 \item $(A,A^{*},-R_{\prec_A}^{*},-L_{\succ_A}^{*},-R_{\prec_{A^{*}}}^{*},-L_{\succ_{A^{*}}}^{*})$ is a matched pair of
associative algebras.
 \end{enumerate}
\end{thm}
\begin{proof}
$(a)\Longrightarrow (b)$ In the light of Proposition \ref{M0}, there are linear maps 
$l_{\succ_A},r_{\succ_A},l_{\prec_A},r_{\prec_A}:A\longrightarrow \hbox{End}(A^{*})$ and
$l_{\succ_{A^{*}}},r_{\succ_{A^{*}}},l_{\prec_{A^{*}}},r_{\prec_{A^{*}}}:A^{*}\longrightarrow \hbox{End}(A)$ such that
$(A,A^{*},l_{\succ_A},r_{\succ_A},l_{\prec_A},r_{\prec_A}, l_{\succ_{A^{*}}},r_{\succ_{A^{*}}},l_{\prec_{A^{*}}},r_{\prec_{A^{*}}})$ 
is a matched pair of
anti-dendriform algebras and
\begin{equation}
x\succ b=r_{\succ_{A^{*}}}(b)x+l_{\succ_A}(x)b,\ \ \ b\succ x=l_{\succ_{A^{*}}}(b)x+r_{\succ_A}(x)b,
\end{equation}
\begin{equation}
x\prec b=r_{\prec_{A^{*}}}(b)x+l_{\prec_A}(x)b,\ \ \ b\succ x=l_{\prec_{A^{*}}}(b)x+r_{\prec_A}(x)b,
\end{equation}
for all $x\in A$ and $b\in A^{*}$.
Then we obtain,
\begin{align*}\langle l_{\succ_A}(x)b,y\rangle &=\omega (x\succ b,y)=-\omega(b,y\cdot x)=-\langle b,y\prec x+y\succ x\rangle
\\&=-\langle b,R_{\succ_A}(x)y+R_{\prec_A}(x)y\rangle\\&=-\langle R_{\succ_A}^{*}(x)b+R_{\prec_A}^{*}(x)b,y \rangle,\end{align*}

which yields that
$l_{\succ_A}=-R_{\succ_A}^{*}-R_{\prec_A}^{*}$.
Analogously, $r_{\succ_A}=L_{\prec_A}^{*},l_{\prec_A}=R_{\succ_A}^{*},r_{\prec_A}=-(L_{\prec_A}^{*}+L_{\succ_A}^{*}),
l_{\succ_{A^{*}}}=-R_{\succ_{A^{*}}}^{*}-R_{\prec_{A^{*}}}^{*}
,r_{\succ_{A^{*}}}=L_{\prec_{A^{*}}}^{*},l_{\prec_{A^{*}}}=R_{\succ_{A^{*}}}^{*},r_{\prec_{A^{*}}}=-(L_{\prec_{A^{*}}}^{*}+L_{\succ_{A^{*}}}^{*})
$. Thus, (b) holds.

$(b)\Longrightarrow (c)$ It can be obtained by Corollary \ref {Ma}.

$(c)\Longrightarrow (a)$  Assume that $(A,A^{*},-R_{\prec_A}^{*},-L_{\succ_A}^{*},-R_{\prec_{A^{*}}}^{*},-L_{\succ_{A^{*}}}^{*})$ is
a matched pair of associative algebras. Then $(A\bowtie A^{*},\cdot)$ is an associative algebra with $\cdot$ defined as follows: 
\begin{equation*}x\cdot a=-L_{\succ_{A^{*}}}^{*}(a) x-R_{\prec_A}^{*}(x)a,\ \ \ 
a\cdot x=-R_{\prec_{A^*}}^{*}(a) x-L_{\succ_A}^{*}(x)a, \ \ \forall ~x\in A, a\in A^{*}.\end{equation*}
According to Eq. \eqref{Co1}, we have
\begin{align*}&\omega(x\cdot a,y)+\omega(a\cdot y,x)+\omega(y\cdot x,a)\\
=&\omega(-L_{\succ_{A^{*}}}^{*}(a) x-R_{\prec_A}^{*}(x)a,y)+\omega(-R_{\prec_{A^*}}^{*}(a) y-L_{\succ_A}^{*}(y)a,x)+\omega(y\cdot x,a)
\\=&\langle -R_{\prec_A}^{*}(x)a,y\rangle+\langle -L_{\succ_A}^{*}(y)a,x\rangle+\langle y\cdot x,a\rangle
\\=&\langle a,-y\prec_A x\rangle+\langle a,-y\succ_A x\rangle+\langle y\cdot x,a\rangle\\&=0,\end{align*}
for all $x,y\in A, a\in A^{*}$, 
 which implies that $\omega$ defined by Eq. \eqref{Co1}
is a commutative Connes cocycle.
Thus, $(A\oplus A^{*},\omega,A, A^{*})$ is a double construction of the commutative Connes cocycle.
\end{proof}

\begin{defi} \label{DB1} An anti-dendriform coalgebra is a triple $(A,\Delta_{\succ},\Delta_{\prec})$, where
$A$ is a vector space and $\Delta_{\succ},\Delta_{\prec}:A\longrightarrow A\otimes A$ are linear maps such that
the following conditions hold:
\begin{equation}\label{Ca1}
( \Delta_{\succ}\otimes I)\Delta_{\prec}=(I\otimes\Delta_{\prec})\Delta_{\succ},\end{equation}
\begin{equation}\label{Ca2}
(I\otimes \Delta_{\succ})\Delta_{\succ}=-(\Delta\otimes I)\Delta_{\succ}=
( \Delta_{\prec}\otimes I)\Delta_{\prec}=-(I\otimes\Delta)\Delta_{\prec},\end{equation}
where $\Delta=\Delta_{\succ}+\Delta_{\prec}.$
\end{defi}

\begin{defi} \label{DB2} An anti-dendriform D-bialgebra is a quintuple $(A,\succ,\prec,\Delta_{\succ},\Delta_{\prec})$, 
such that $(A,\succ,\prec)$ is an anti-dendriform algebra, $(A,\Delta_{\succ},\Delta_{\prec})$ is 
an anti-dendriform coalgebra, and the following compatible conditions hold:
\begin{equation}\label{D1}\Delta_{\prec}(x\cdot y)=(R_{\cdot}(y)\otimes I)\Delta_{\prec}(x)-(I\otimes L_{\succ} (x))\Delta_{\prec}(y),\end{equation}
\begin{equation}\label{D2}\Delta_{\succ}(x\cdot y)=(I\otimes L_{\cdot}(x))\Delta_{\succ}(y)-(R_{\prec} (y)\otimes I)\Delta_{\succ}(x),\end{equation}
\begin{equation}\label{D3}
(I\otimes R_{\cdot}(y))\Delta_{\succ}(x)+(L_{\succ} (y)\otimes I)\Delta_{\succ}(x)-\tau(I\otimes R_{\prec} (x))\Delta_{\prec}(y)-\tau
(L_{\cdot}(x)\otimes I)\Delta_{\prec}(y)=0,\end{equation}
\begin{equation}\label{D4}\Delta(x\prec y)=(R_{\prec} (y)\otimes I)\Delta(x)-(I\otimes L_{\prec} (x))\Delta_{\succ}(y),\end{equation}
\begin{equation}\label{D5}\Delta(x\succ y)=(I\otimes L_{\succ} (x))\Delta(y)-(R_{\succ} (y)\otimes I)\Delta_{\prec}(x),\end{equation}
\begin{equation} \label{D6}(L_{\succ} (x)\otimes I)\Delta(y)+(R_{\succ} (y)\otimes I)\tau\Delta_{\succ}(x)
-(I\otimes L_{\prec} (y))\tau\Delta_{\prec}(x)-(I\otimes R_{\prec} (x))\Delta(y)=0,\end{equation}
where $\cdot=\succ+\prec,~\Delta=\Delta_{\succ}+\Delta_{\prec}$ and $R_{\cdot}=R_{\prec}+R_{\succ},~L_{\cdot}=L_{\prec}+L_{\succ}$.
\end{defi}

\begin{thm}
Let $(A,\succ_A,\prec_A)$ be an anti-dendriform algebra. Assume that there is an anti-dendriform algebra
 structure $(\succ_{A^{*}},\prec_{A^{*}})$
on the dual space $A^{*}$. Let $(A,\cdot)$ and $(A^{*},\ast)$ be the associated associative algebra of $(A,\succ_A,\prec_A)$
and $(A^{*},\succ_{A^{*}},\prec_{A^{*}})$ respectively. Suppose that $\Delta_{\succ},\Delta_{\prec}:A\longrightarrow A\otimes A$
and $\delta_{\succ},\delta_{\prec}:A^{*}\longrightarrow A^{*}\otimes A^{*}$ are the dual of $\succ_{A^{*}},\prec_{A^{*}}$
 and $\succ_A,\prec_A$ respectively. Then the following conditions are equivalent:
 \begin{enumerate}
\item $(A,\succ_A,\prec_A,\Delta_{\succ},\Delta_{\prec})$ is an anti-dendriform D-bialgebra.
\item Eqs.\eqref{D1}-\eqref{D3} hold and the following equations hold:
\begin{equation}\label{D7}\delta_{\prec}(a\ast b)=(R_{\ast}(b)\otimes I)\delta_{\prec}(a)-(I\otimes L_{\succ_{A^{*}}} (a))\delta_{\prec}(b),\end{equation}
\begin{equation}\label{D8}\delta_{\succ}(a\ast b)=(I\otimes L_{\ast}(a))\delta_{\succ}(b)-(R_{\prec_{A^{*}}} (b)\otimes I)\delta_{\succ}(a),\end{equation}
\begin{equation}\label{D9}
(I\otimes R_{\ast}(b))\delta_{\succ}(a)+(L_{\succ_{A^{*}}} (b)\otimes I)\delta_{\succ}(a)-\tau(I\otimes R_{\prec_{A^{*}}} (a))
\delta_{\prec}(b)-\tau(L_{\ast}(a)\otimes I)\delta_{\prec}(b)=0.\end{equation}
\item $(A,A^{*},-R_{\prec_A}^{*},-L_{\succ_A}^{*},-R_{\prec_{A^*}}^{*},-L_{\succ_{A^*}}^{*})$ is a matched pair of
associative algebras,
 \end{enumerate}
 where $L_{\ast}=L_{\prec_{A^{*}}}+L_{\succ_{A^*}},~~R_{\ast}=R_{\prec_{A^*}}+R_{\succ_{A^*}},~~\ast=\succ_{A^*}+\prec_{A^*}$ and $\Delta=\Delta_{\prec}+\Delta_{\succ}$.
\end{thm}
\begin{proof} $(a)\Longleftrightarrow (b)$ 
For all $x,y\in A$ and $a,b\in A^{*}$,
\begin{align*}&\langle(R_{\ast}(b)\otimes I)\delta_{\prec}(a)-(I\otimes L_{\succ_{A^*}} (a))\delta_{\prec}(b)-\delta_{\prec}(a\ast b),x\otimes y\rangle
\\=&\langle a,\delta_{\prec}^{*}(R_{\ast}^{*}(b)\otimes I)(x\otimes y)\rangle-
\langle b,\delta_{\prec}^{*}(I\otimes L_{\succ_{A^*}}^{*} (a))(x\otimes y)\rangle-\langle a\ast b,\delta_{\prec}^{*}(x\otimes y)\rangle
\\=&\langle a,(R_{\ast}^{*}(b)x)\prec_{A} y\rangle-\langle b,x\prec_A (L_{\succ_{A^*}}^{*}(a) (y))\rangle-\langle a\ast b,x\prec_A y\rangle
\\=&\langle a,R_{\prec_A} (y)(R_{\ast}^{*}(b)x)\rangle-\langle b,L_{\prec_A}(x) (L_{\succ_{A^*}}^{*}(a)y)\rangle-\langle a\otimes b,\Delta(x\prec_A y)\rangle
\\=&\langle R_{\ast}(b)(R_{\prec_A}^{*} (y)a),x\rangle-\langle L_{\succ_{A^*}}(a)(L_{\prec}^{*}(x)b),  y\rangle
-\langle a\otimes b,\Delta(x\prec_A y)\rangle
\\=&\langle (R_{\prec_A}^{*} (y)a)\ast b,x\rangle-\langle a\succ_{A^{*}}(L_{\prec_A}^{*}(x)b),  y\rangle
-\langle a\otimes b,\Delta(x\prec_A y)\rangle
\\=&\langle (R_{\prec_A}^{*} (y)a)\otimes b,\Delta(x)\rangle
-\langle a\otimes(L_{\prec_A}^{*}(x)b),  \Delta_{\succ}(y)\rangle-\langle a\otimes b,\Delta(x\prec_A y)\rangle
\\=&\langle a\otimes b,(R_{\prec_A} (y)\otimes I)\Delta(x)
-(I\otimes L_{\prec_A}(x))\Delta_{\succ}(y)-\Delta(x\prec_A y)\rangle,
\end{align*}
which means that Eq.\eqref{D7} holds if and only if Eq.\eqref{D4} holds. Take the same procedure, we can verify that
Eq.\eqref{D8} holds if and only if Eq.\eqref{D5} holds, Eq.\eqref{D9} holds if and only if Eq.\eqref{D6} holds.

$(a)\Longleftrightarrow (c)$ We need to prove that Eqs.\eqref{AM1}-\eqref{AM6} are equivalent to
 Eqs.\eqref{D1}-\eqref{D6}. In fact, let $l_{1}=-R_{\prec_A}^{*},r_{2}=-L_{\succ_{A^*}}^{*}$,
 for all $x,y\in A$ and $a,b\in A^{*}$, we obtain
\begin{align*}&\langle R_{\prec_A}^{*}(L_{\succ_{A^*}}^{*}(a)x)b-(R_{\prec_A}^{*}(x)a)\ast b+R_{\prec_A}^{*}(x)(a\ast b),y\rangle
\\=&\langle b,y\prec_A (L_{\succ_{A^*}}^{*}(a)x)\rangle-
\langle R_{\ast}(b)(R_{\prec_A}^{*}(x)a),y\rangle+
\langle a\ast b,R_{\prec_A}(x)y\rangle
\\=&\langle b,L_{\prec_A}(y) (L_{\succ_{A^*}}^{*}(a)x)\rangle-
\langle a,R_{\prec_A}(x)(R_{\ast}^{*}(b)y)\rangle+
\langle a\otimes b,\Delta(z\prec_Ax)\rangle
\\=&\langle L_{\succ_{A^*}}(a)L_{\prec_A}^{*}(y)b, x\rangle-
\langle R_{\ast}(b)R_{\prec_A}^{*}(x)a,y\rangle+
\langle a\otimes b,\Delta(y\prec_Ax)\rangle
\\=&\langle a\succ_{A^*}L_{\prec_A}^{*}(y)b, x\rangle-
\langle (R_{\prec_A}^{*}(x)a)\ast b,y\rangle+
\langle a\otimes b,\Delta(y\prec_Ax)\rangle
\\=&\langle a\otimes L_{\prec_A}^{*}(y)b, \Delta_{\succ}(x)\rangle-
\langle (R_{\prec_A}^{*}(x)a)\otimes b,\Delta(y)\rangle+
\langle a\otimes b,\Delta(y\prec_Ax)\rangle
\\=&\langle a\otimes b, (I\otimes L_{\prec_A}(y))\Delta_{\succ}(x)-(R_{\prec_A}(x)\otimes I)\Delta(y)+\Delta(y\prec_Ax)\rangle,\end{align*}
which implies that Eq.\eqref{AM1} holds if and only if Eq.\eqref{D4} holds. The proof of other cases is similar.
\end{proof}
%%%%%%%%%%%%%%%%%%%%%%%%%%%%%%%%%%%%%%%%%%%%%%%%%%%%%%%%%%%%%%%%%%%%%%%%%%%%%%%%%%%%%%%%%%%%%%%%%%%%%%%%%%
\subsection{Coboundary anti-dendriform D-bialgebra}
\begin{defi} 
An anti-dendriform D-bialgebra $(A,\succ,\prec,\Delta_{\succ},\Delta_{\prec})$ is called coboundary 
 if $\Delta_{\succ},\Delta_{\prec}$ are defined by the following equations respectively:
 \begin{align}&\label{CD1}
\Delta_{\succ}(x)=-(R_{\prec}(x)\otimes I+I\otimes L_{\cdot}(x))r_{\succ},
\\&\label{CD2}\Delta_{\prec}(x)=(R_{\cdot}(x)\otimes I+I\otimes L_{\succ}(x))r_{\prec}
\end{align}
 for all $x\in A$, where $r_{\prec},r_{\succ}\in A\otimes A$.
\end{defi}
It is straightforward to show that Eqs.\eqref{D1}-\eqref{D2} hold if $\Delta_{\succ},\Delta_{\prec}$ 
are given by Eqs.\eqref{CD1}-\eqref{CD2}.
 
 \begin{pro} \label{DB3}
 Let $(A,\succ,\prec)$ be an anti-dendriform algebra and $r_{\prec},r_{\succ}\in A\otimes A$. Define
 $\Delta_{\succ},\Delta_{\prec}$ by Eqs.\eqref{CD1}-\eqref{CD2}. Then
 \begin{enumerate}
		\item Eq.\eqref{D3} holds if and only if the following equation holds:
\begin{equation}\label{CD3}(R_{\prec}(x)\otimes I+I\otimes L_{\cdot}(x))(L_{\succ}(y)\otimes I+I\otimes R_{\cdot}(y))(r_{\succ}+\tau r_{\prec})=0.\end{equation}
\item Eq.\eqref{D4} holds if and only if the following equation holds:
\begin{equation}\label{CD4}[I\otimes L_{\succ}( x\prec y)-R_{\prec}( y)\otimes L_{\succ}( x)+R_{\prec}(x\prec y+x\cdot y)\otimes I](r_{\succ}-r_{\prec})=0,\end{equation}
\item Eq.\eqref{D5} holds if and only if the following equation holds:
\begin{align}\label{CD5}
[I\otimes L_{\succ}( x\succ y+x\cdot y)+R_{\prec}( x\succ y)\otimes I-R_{\prec}(y)\otimes L_{\succ}( x)
(r_{\succ}-r_{\prec})=0
\end{align}
\item Eq.\eqref{D6} holds if and only if the following equation holds:
\begin{align}\label{CD6}&[L_{\succ}( x)R_{\succ}(y)\otimes I-R_{\succ}(y)\otimes R_{\prec}( x)](r_{\prec}+\tau r_{\succ})\\&+
[I\otimes R_{\prec}(x)L_{\prec}(y)-L_{\succ}(x)\otimes L_{\prec}( y)]( r_{\succ}+\tau r_{\prec}) \nonumber
\\&+[ L_{\succ}( x)R_{\prec}(y)\otimes I-R_{\prec}(y)\otimes R_{\prec}( x)
+L_{\succ}( x)\otimes L_{\succ}( y)-I\otimes R_{\prec}( x)L_{\succ}( y)](r_{\succ}-r_{\prec})
=0.\nonumber\end{align}
\end{enumerate}
 \end{pro}
 \begin{proof} (a)
By Eqs.\eqref{R6}, \eqref{R7} and Eqs.\eqref{CD1}-\eqref{CD2}, we obtain
\begin{align*}
&(I\otimes R_{\cdot}(y))\Delta_{\succ}(x)+(L_{\succ_{A}} (y)\otimes I)\Delta_{\succ}(x)-\tau(I\otimes R_{\prec_A} (x))\Delta_{\prec}(y)-\tau
(L_{\cdot}(x)\otimes I)\Delta_{\prec}(y)\\
=&-(I\otimes R_{\cdot}(y))(R_{\prec}(x)\otimes I+I\otimes L_{\cdot}(x))r_{\succ}
-(L_{\succ_{A}} (y)\otimes I)(R_{\prec}(x)\otimes I+I\otimes L_{\cdot}(x))r_{\succ}
\\&-\tau(I\otimes R_{\prec_A} (x))(R(y)\otimes I+I\otimes L_{\succ}(y))r_{\prec}
-\tau(L_{\cdot}(x)\otimes I)(R(y)\otimes I+I\otimes L_{\succ}(y))r_{\prec}
\\=&-(R_{\prec}(x)\otimes R_{\cdot}(y)+I\otimes R_{\cdot}(y)L(x))r_{\succ}
-(L_{\succ_{A}} (y)R_{\prec}(x)\otimes I+L_{\succ_{A}} (y)\otimes L_{\cdot}(x))r_{\succ}
\\&-\tau(R_{\cdot}(y)\otimes R_{\prec_A} (x)+I\otimes R_{\prec_A} (x)L_{\succ}(y))r_{\prec}
-\tau(L_{\cdot}(x)R_{\cdot}(y)\otimes I+L_{\cdot}(x)\otimes L_{\succ}(y))r_{\prec}
\\=&-(R_{\prec}(x)\otimes I+I\otimes L_{\cdot}(x))(L_{\succ}(y)\otimes I+I\otimes R_{\cdot}(y))(r_{\succ}+\tau r_{\prec}),\end{align*}
which implies that Eq.\eqref{D3} holds if and only if Eq.\eqref{CD3} holds.

(b) By Eqs.\eqref{R1}, \eqref{R2} and \eqref{R5}, we have
 \begin{align*}&(R_{\prec} (y)\otimes I)\Delta(x)-(I\otimes L_{\prec_A} (x))\Delta_{\succ}(y)-\Delta(x\prec y)\\
=&(R_{\prec} (y)\otimes I)(R_{\cdot}(x)\otimes I+I\otimes L_{\succ}(x))r_{\prec}
-(R_{\prec} (y)\otimes I)(R_{\prec}(x)\otimes I+I\otimes L_{\cdot}(x))r_{\succ}
\\&+(I\otimes L_{\prec_A} (x))(R_{\prec}(y)\otimes I+I\otimes L_{\cdot}(y))r_{\succ}
-(R_{\cdot}(x\prec y)\otimes I\\&+I\otimes L_{\succ}(x\prec y))r_{\prec}
+(R_{\prec}(x\prec y)\otimes I+I\otimes L_{\cdot}(x\prec y))r_{\succ}\\=&
[R_{\prec}( y)\otimes L_{\prec}( x)+I\otimes L_{\prec}( x) L_{\cdot}( y)-R_{\prec}( y)R_{\prec}( x)\otimes I
-R_{\prec}( y)\otimes L_{\cdot}( x)+R_{\prec}(x\prec y)\otimes I\\&+I\otimes L_{\cdot}(x\prec y)]r_{\succ}
+[R_{\prec}( y)R_{\cdot}(x)\otimes I+R_{\prec}( y)\otimes L_{\succ}( x)-R_{\cdot}( x\prec y)\otimes I-I\otimes L_{\succ}( x\prec y)]r_{\prec}
\\=&[I\otimes L_{\succ}( x\prec y)-R_{\prec}( y)\otimes L_{\succ}( x)+R_{\prec}(x\prec y+x\cdot y)\otimes I](r_{\succ}-r_{\prec}),
\end{align*}
which yields that Eq.\eqref{D4} holds if and only if Eq.\eqref{CD4} holds.
The other cases can be checked similarly.
 \end{proof}

\begin{pro} \label{DB4}Let $(A,\succ,\prec)$ be an anti-dendriform algebra and 
$r_{\prec}=\sum_{i}a_i\otimes b_i,r_{\succ}=\sum_{i}c_i\otimes d_i\in A\otimes A$. Assume that
$\Delta_{\succ},\Delta_{\prec}$ defined by Eqs.\eqref{CD1}-\eqref{CD2}. Then
\begin{enumerate}
		\item Eq. \eqref{Ca1} holds if and only if the following equation holds:
\begin{align}\label{CD7}&(R_{\prec}(x)\otimes I \otimes I)(r_{\succ,12}\prec r_{\prec,13}+r_{\prec,23}\cdot r_{\succ,12}+r_{\succ,13}\succ r_{\prec,23})
\\&-(I\otimes I\otimes L_{\succ}(x))(r_{\succ,12}\prec r_{\prec,13}+r_{\prec,23}\cdot r_{\succ,12}+r_{\succ,13}\succ r_{\prec,23})=0\nonumber
,\end{align}
\item $(I\otimes \Delta_{\succ})\Delta_{\succ}=-(\Delta\otimes I)\Delta_{\succ}$ holds if and only if the following equation holds:
\begin{align}\label{CD8}&(r_{\succ,12}-r_{\prec,12})\prec((R_{\prec}(x)\otimes I \otimes I))r_{\succ,13}
+(I\otimes R_{\prec}(x)\otimes I)r_{\succ,23}\succ(r_{\succ,12}-r_{\prec,12})
\\&+(I\otimes I\otimes L_{\cdot}(x))(r_{\succ,13}\cdot r_{\succ,23}-r_{\prec,12}\cdot r_{\succ,13}-r_{\succ,23}\succ r_{\prec,12}
+r_{\succ,12}\prec r_{\succ,13}+r_{\succ,23}\cdot r_{\succ,12})
\nonumber\\&
+(R_{\prec}(x)\otimes I \otimes I)(r_{\succ,23}\prec r_{\succ,12}+r_{\succ,13}\cdot r_{\succ,23}-r_{\prec,12}\succ r_{\succ,13})=0\nonumber,\end{align}
\item $(I\otimes \Delta_{\succ})\Delta_{\succ}=-(\Delta\otimes I)\Delta_{\succ}$ holds if and only if the following equation holds:
\begin{align}\label{CD9}&
(R_{\cdot}(x)\otimes I \otimes I)(r_{\prec,12}\cdot r_{\prec,13}-
r_{\succ,23}\prec r_{\prec,12}-r_{\prec,13}\cdot r_{\succ,23}+r_{\prec,23}\cdot r_{\prec,12}+r_{\prec,13}\succ r_{\prec,23})\\&+(I\otimes I\otimes L_{\succ}(x))(r_{\prec,12}\cdot r_{\prec,13}+
r_{\prec,23}\succ r_{\prec,12}-r_{\prec,13}\prec r_{\succ,23})\nonumber
\\&+(I\otimes I\otimes L_{\succ}(x))r_{\prec,13}\succ (r_{\prec,23}- r_{\succ,23})
+(r_{\prec,23}- r_{\succ,23})\prec (I\otimes L_{\succ}(x)\otimes I)r_{\prec,12}
=0\nonumber,\end{align}
\item $(I\otimes \Delta_{\succ})\Delta_{\succ}(x)=( \Delta_{\prec}\otimes I)\Delta_{\prec}(x)$ holds if and only if the following equation holds:
\begin{align}\label{CD10}&
 (R_{\prec}(x)\otimes I \otimes I)(r_{\succ,23}\prec r_{\succ,12}+r_{\succ,13}\cdot r_{\succ,23}-r_{\prec,12}\succ r_{\prec,13})
\\&-(I\otimes I\otimes L_{\succ}(x))(r_{\prec,12}\cdot r_{\prec,13}+r_{\prec,23}\succ r_{\prec,12}-r_{\succ,13}\prec r_{\succ,23})\nonumber
\\&-[(I\otimes R_{\prec}(x)\otimes I)r_{\succ,23}]\prec r_{\succ,12}
+[(I\otimes R_{\prec}(x)\otimes I)r_{\prec,23}]\prec r_{\prec,12}
=0\nonumber.\end{align}
\end{enumerate}
\end{pro}

\begin{proof} In view of Eqs.\eqref{A1}, \eqref{CD1}-\eqref{CD2}, we get
\begin{align*}
&( \Delta_{\succ}\otimes I)\Delta_{\prec}-(I\otimes\Delta_{\prec})\Delta_{\succ}
\\=&( \Delta_{\succ}\otimes I)(R_{\cdot}(x)\otimes I+I\otimes L_{\succ}(x))\sum_{i}a_i\otimes b_i+
(I\otimes\Delta_{\prec})(R_{\prec}(x)\otimes I+I\otimes L_{\cdot}(x))\sum_{i}c_i\otimes d_i
\\=&( \Delta_{\succ}\otimes I)\sum_{i}[(a_i\cdot x)\otimes b_i+a_i\otimes(x\succ b_i)]+
(I\otimes\Delta_{\prec})\sum_{i}[(c_i\prec x)\otimes d_i+c_i\otimes(x\cdot d_i)]
\\=&\sum_{i}\Delta_{\succ}(a_i\cdot x)\otimes b_i+\sum_{i}\Delta_{\succ}(a_i)\otimes (x\succ b_i)
+\sum_{i} ( c_i\prec x)\otimes \Delta_{\prec}(d_i)+\sum_{i}c_i\otimes \Delta_{\prec}(x\cdot d_i)
\\=&-\sum_{i,j}(c_j\prec(a_i\cdot x))\otimes d_j\otimes b_i-
\sum_{i,j}c_j\otimes[(a_i\cdot x)\cdot d_j]\otimes b_i
-\sum_{i,j}[(c_j\prec a_i)\otimes d_j\\&+c_j\otimes (a_i \cdot d_j)]\otimes(x\succ b_i)
+\sum_{i,j}(c_j\prec x)\otimes [ (a_i\cdot d_j) \otimes b_i+a_i\otimes (d_j\succ b_i)]
\\&+\sum_{i,j}c_j\otimes [(a_i\cdot(x\cdot d_j))\otimes b_i+a_i\otimes((x\cdot d_j)\succ b_i)]
\\=&(R_{\prec}(x)\otimes I \otimes I)(r_{\succ,12}\prec r_{\prec,13})
-(I\otimes I\otimes L_{\succ}(x))(r_{\succ,12}\prec r_{\prec,13})
-(I\otimes I\otimes L_{\succ}(x))(r_{\prec,23}\cdot r_{\succ,12})
\\&+(R_{\prec}(x)\otimes I \otimes I)(r_{\prec,23}\cdot r_{\succ,12})
+(R_{\prec}(x)\otimes I \otimes I)(r_{\succ,13}\succ r_{\prec,23})
-(I\otimes I\otimes L_{\succ}(x))(r_{\succ,13}\succ r_{\prec,23})
\\=&(R_{\prec}(x)\otimes I \otimes I)(r_{\succ,12}\prec r_{\prec,13}+r_{\prec,23}\cdot r_{\succ,12}+r_{\succ,13}\succ r_{\prec,23})
\\&-(I\otimes I\otimes L_{\succ}(x))(r_{\succ,12}\prec r_{\prec,13}+r_{\prec,23}\cdot r_{\succ,12}+r_{\succ,13}\succ r_{\prec,23})
,\end{align*}
which indicates that Eq.\eqref{Ca1} holds if and only if Eq.\eqref{CD7} holds.
Thanks to Eq.\eqref{A1}, we get
\begin{align}&\label{YE1}(I\otimes I\otimes L_{\prec}(x))(r_{\succ,13}\cdot r_{\succ,23})
+(I\otimes I\otimes L_{\succ}(x))(r_{\succ,13}\succ r_{\succ,23})=0,
\\&\label{YE2}(R_{\succ}(x)\otimes I \otimes I)(r_{\prec,12}\cdot r_{\prec,13})
+(R_{\prec}(x)\otimes I \otimes I)(r_{\prec,12}\prec r_{\prec,13})=0.\end{align}
Using Eqs.\eqref{CD1}-\eqref{CD2} and \eqref{YE1}-\eqref{YE2}, we obtain
\begin{align*}
&(I\otimes \Delta_{\succ})\Delta_{\succ}(x)-( \Delta_{\prec}\otimes I)\Delta_{\prec}(x)
\\=&\sum_{i,j}(c_i\prec x)\otimes (c_j\prec d_i)\otimes d_j+(c_i\prec x)\otimes c_j\otimes( d_i\cdot d_j)
+c_i\otimes( c_j\prec(x\cdot d_i))\otimes d_j
\\&+c_i\otimes c_j\otimes(x\cdot d_i))\cdot d_j
-(a_j\cdot(a_i\cdot x))\otimes b_j\otimes b_i-a_j\otimes ((a_i\cdot x)\succ b_j)\otimes b_i
\\&-(a_j\cdot a_i\otimes b_j\otimes(x\succ b_i)-a_j\otimes (a_i\succ b_j)\otimes (x\succ b_i)
\\=&(R_{\prec}(x)\otimes I \otimes I)(r_{\succ,23}\prec r_{\succ,12})
+(R_{\prec}(x)\otimes I \otimes I)(r_{\succ,13}\cdot r_{\succ,23})
\\&-[(I\otimes R_{\prec}(x)\otimes I)r_{\succ,23}]\prec r_{\succ,12}
+(I\otimes I\otimes L_{\cdot}(x))(r_{\succ,13}\cdot r_{\succ,23})
-(r_{\succ,13}\cdot r_{\succ,23})
\\&-(R_{\cdot}(x)\otimes I \otimes I)(r_{\prec,12}\cdot r_{\prec,13})
+[(I\otimes R_{\prec}(x)\otimes I)r_{\prec,23}]\prec r_{\prec,12}
\\&-(I\otimes I\otimes L_{\succ}(x))(r_{\prec,12}\cdot r_{\prec,13})
-(I\otimes I\otimes L_{\succ}(x))(r_{\prec,23}\succ r_{\prec,12})
\\=& (R_{\prec}(x)\otimes I \otimes I)(r_{\succ,23}\prec r_{\succ,12}+r_{\succ,13}\cdot r_{\succ,23}-r_{\prec,12}\succ r_{\prec,13})
\\&-(I\otimes I\otimes L_{\succ}(x))(r_{\prec,12}\cdot r_{\prec,13}+r_{\prec,23}\succ r_{\prec,12}-r_{\succ,13}\prec r_{\succ,23})
\\&-[(I\otimes R_{\prec}(x)\otimes I)r_{\succ,23}]\prec r_{\succ,12}
+[(I\otimes R_{\prec}(x)\otimes I)r_{\prec,23}]\prec r_{\prec,12}
\\&+(I\otimes I\otimes L_{\prec}(x))(r_{\succ,13}\cdot r_{\succ,23})
+(I\otimes I\otimes L_{\succ}(x))(r_{\succ,13}\succ r_{\succ,23})
\\&-(R_{\succ}(x)\otimes I \otimes I)(r_{\prec,12}\cdot r_{\prec,13})
-(R_{\prec}(x)\otimes I \otimes I)(r_{\prec,12}\prec r_{\prec,13})
\\=& (R_{\prec}(x)\otimes I \otimes I)(r_{\succ,23}\prec r_{\succ,12}+r_{\succ,13}\cdot r_{\succ,23}-r_{\prec,12}\succ r_{\prec,13})
\\&-(I\otimes I\otimes L_{\succ}(x))(r_{\prec,12}\cdot r_{\prec,13}+r_{\prec,23}\succ r_{\prec,12}-r_{\succ,13}\prec r_{\succ,23})
\\&-[(I\otimes R_{\prec}(x)\otimes I)r_{\succ,23}]\prec r_{\succ,12}
+[(I\otimes R_{\prec}(x)\otimes I)r_{\prec,23}]\prec r_{\prec,12}
,\end{align*}
which yields that the item (d) holds.
The remaining part can be verified analogously.
\end{proof}

\begin{thm} \label{YE3} Let $(A,\succ,\prec)$ be an anti-dendriform algebra and 
$r_{\prec}=\sum_{i}a_i\otimes b_i,~~r_{\succ}=\sum_{i}c_i\otimes d_i\in A\otimes A$. Assume that
$\Delta_{\succ},\Delta_{\prec}$ defined by Eqs.\eqref{CD1}-\eqref{CD2}.
Then $(A,\succ,\prec,\Delta_{\succ},\Delta_{\prec})$ is an anti-dendriform D-bialgebra if and only if Eqs.\eqref{CD3}-\eqref{CD10} hold.\end{thm}
\begin{proof} This follows from Definition \ref{DB1}, Definition \ref{DB2},Proposition \ref{DB3} and Proposition \ref{DB4}.\end{proof}

Eqs.\eqref{CD3}-\eqref{CD10} are very complicated. We study some simple and special case:
$r_{\prec}=r_{\succ}=r\in A\otimes A$ and $r$ is skew-symmetric.

The following conclusion is apparently.

\begin{pro} \label{YE4} Assume that $r_{\prec}=r_{\succ}=r\in A\otimes A$ and $r$ is skew symmetric. Then
\begin{enumerate}
		\item Eqs.\eqref{CD3}-\eqref{CD6} hold automatically.
\item Eq.\eqref{CD7} holds if and only if
\begin{align*}(R_{\prec}(x)\otimes I \otimes I-I\otimes I\otimes L_{\succ}(x))(r_{12}\prec r_{13}+r_{23}\cdot r_{12}+r_{13}\succ r_{23})
=0,\end{align*}
\item Eq.\eqref{CD8} holds if and only if
\begin{align*}&
(R_{\prec}(x)\otimes I \otimes I+I\otimes I\otimes L_{\cdot}(x))(r_{23}\prec r_{12}+r_{13}\cdot r_{23}-r_{12}\succ r_{13})=0,\end{align*}
\item Eq.\eqref{CD9} holds if and only if
\begin{align*}
(R_{\cdot}(x)\otimes I \otimes I+I\otimes I\otimes L_{\succ}(x))(r_{12}\cdot r_{13}+
r_{23}\succ r_{12}-r_{13}\prec r_{23})
=0,\end{align*}
\item Eq.\eqref{CD10} holds if and only if
\begin{align*}
 (R_{\prec}(x)\otimes I \otimes I)(r_{23}\prec r_{12}+r_{13}\cdot r_{23}-r_{12}\succ r_{13})
=(I\otimes I\otimes L_{\succ}(x))(r_{12}\cdot r_{13}+r_{23}\succ r_{12}-r_{13}\prec r_{23}).\end{align*}
\end{enumerate}
\end{pro}

\begin{rmk} \label{YE5} Assume that $\sigma_{123},\sigma_{132}:A\otimes A\otimes A\longrightarrow A\otimes A\otimes A$ 
are linear maps defined respectively by
\begin{equation*}\sigma_{123}(x\otimes y\otimes z)=z\otimes x\otimes y,
~~\sigma_{132}(x\otimes y\otimes z)=y\otimes z\otimes x,~~\forall~x,y,z\in A. \end{equation*}
If $r$ is skew-symmetric, then
\begin{align*}&\sigma_{123}( r_{12}\cdot r_{13}+r_{23}\succ r_{12}-r_{13}\prec r_{23})=-(r_{12}\prec r_{13}+r_{23}\cdot r_{12}+r_{13}\succ r_{23}),
\\&
\sigma_{132}( r_{12}\cdot r_{13}+r_{23}\succ r_{12}-r_{13}\prec r_{23})=r_{23}\prec r_{12}+r_{13}\cdot r_{23}-r_{12}\succ r_{13}.\end{align*}
\end{rmk}

\begin{pro} Let $r_{\prec}=r_{\succ}=r\in A\otimes A$ and $r$ be skew-symmetric.
Assume that
$\Delta_{\succ},\Delta_{\prec}$ defined by Eqs.\eqref{CD1}-\eqref{CD2}.
Then $(A,\succ,\prec,\Delta_{\succ},\Delta_{\prec})$ is an anti-dendriform D-bialgebra if and only if the
following equation holds:
\begin{equation} \label{YE6} r_{12}\cdot r_{13}+r_{23}\succ r_{12}-r_{13}\prec r_{23}=0.\end{equation}
\end{pro}
\begin{proof} 
Combining Theorem \ref{YE3}, Proposition \ref{YE4} and Remark \ref{YE5}, we can get the statement.
\end{proof}

\begin{defi} Let $(A,\succ,\prec)$ be an anti-dendriform algebra and 
$r\in A\otimes A$. Eq. \eqref{YE6} is called the {\bf anti-dendriform Yang-Baxter equation} or {\bf AD-YBE} in short. 
\end{defi}

For a vector space $A$, the isomorphism $A\otimes A^{*}\simeq Hom (A^{*},A)$ identifies a $r\in A\otimes A$ with a map
$T_{r}:A^{*}\longrightarrow A$. Explicitly, 
\begin{equation} \label{YE7} T_{r}:A^{*}\longrightarrow A,\ \ \ T_{r}(w^{*})=\sum_{i}\langle w^{*},a_i\rangle b_i,
~~\forall~w^{*}\in A^{*},r=\sum_ia_i\otimes b_i.\end{equation}

\begin{thm} \label{YE8} Let $(A,\succ,\prec)$ be an anti-dendriform algebra and 
$r=\sum_{i}a_i\otimes b_i\in A\otimes A$ be skew-symmetric. Then the following conditions are equivalent:
\begin{enumerate}
		\item $r$ is a solution of AD-YBE, that is, 
\begin{equation*}  r_{12}\cdot r_{13}+r_{23}\succ r_{12}-r_{13}\prec r_{23}=0.\end{equation*}
\item The following equation holds:
\begin{align*}T_{r}(u^{*})\cdot T_{r}(v^{*})+T_{r}(R_{\prec}^{*}(T_{r}(u^{*}))v^{*}+L_{\succ}^{*}(T_{r}(v^{*}))u^{*})=0.\end{align*}
\end{enumerate}\end{thm}

\begin{proof} According to Eq. \eqref{YE7}, for all $u^{*},v^{*},w^{*}\in A^{*}$, we have
\begin{align*}
 \langle T_{r}(u^{*})\cdot T_{r}(v^{*}), w^{*}\rangle&=\sum_{i,j}\langle u^{*},a_i\rangle \langle v^{*},a_j\rangle \langle b_i\cdot b_j,w^{*}\rangle
 =\sum_{i,j}\langle w^{*}\otimes u^{*}\otimes v^{*},b_i\cdot b_j\otimes a_i\otimes a_j\rangle
\\&=\sum_{i,j}\langle w^{*}\otimes u^{*}\otimes v^{*},a_i\cdot a_j\otimes b_i\otimes b_j\rangle
=\langle w^{*}\otimes u^{*}\otimes v^{*},r_{12}\cdot r_{13}\rangle,\end{align*}
\begin{align*}
 \langle T_{r}(R_{\prec}^{*}(T_{r}(u^{*}))v^{*}), w^{*}\rangle
 &=\sum_{i}\langle R_{\prec}^{*}(T_{r}(u^{*}))v^{*},a_i\rangle \langle b_i,w^{*}\rangle
 =\sum_{i}\langle v^{*},R_{\prec}(T_{r}(u^{*}))a_i\rangle \langle b_i,w^{*}\rangle
 \\& =\sum_{i,j}\langle u^{*},a_j\rangle \langle v^{*},a_i\prec b_j\rangle \langle b_i,w^{*}\rangle
 =\sum_{i,j}\langle w^{*}\otimes u^{*}\otimes v^{*},b_i \otimes a_j\otimes (a_i\prec b_j)\rangle 
\\&=-\sum_{i,j}\langle w^{*}\otimes u^{*}\otimes v^{*},a_i \otimes a_j\otimes (b_i\prec b_j)\rangle
=-\langle w^{*}\otimes u^{*}\otimes v^{*},a_i \otimes a_j\otimes (b_i\prec b_j)\rangle,\end{align*}
 \begin{align*}
 \langle T_{r}(L_{\succ}^{*}(T_{r}(v^{*}))u^{*}), w^{*}\rangle
 &=\sum_{i}\langle L_{\succ}^{*}(T_{r}(v^{*}))u^{*},a_i\rangle \langle b_i,w^{*}\rangle
 =\sum_{i}\langle u^{*},L_{\succ}(T_{r}(v^{*}))a_i\rangle \langle b_i,w^{*}\rangle
\\& =\sum_{i,j}\langle v^{*},a_j\rangle \langle u^{*},b_j\succ a_i\rangle \langle b_i,w^{*}\rangle
 =\sum_{i,j}\langle w^{*}\otimes u^{*}\otimes v^{*},b_i \otimes (b_j\succ a_i)\otimes a_j\rangle 
 \\&=\sum_{i,j}\langle w^{*}\otimes u^{*}\otimes v^{*},a_i \otimes (a_j\succ b_i)\otimes b_j\rangle
 =\langle w^{*}\otimes u^{*}\otimes v^{*},r_{23}\succ r_{12}\rangle.\end{align*}
 Thus, Item (a) holds if and only if Item (b) holds.
\end{proof}

Recall that an $\mathcal O$-operator $T$ of an associative 
algebra $(A,\cdot)$ associated to a representation $(V,l,r)$ is a linear map $T:V\longrightarrow A$ satisfying
$T(u)\cdot T(v)=T (l(T(u))v+r(T(v))u)$ for all $u,v\in V$.

Theorem \ref{YE8} is rewritten as follows.
\begin{ex}  Let $(A,\succ,\prec)$ be an anti-dendriform algebra and
$r=\sum_{i}a_i\otimes b_i\in A\otimes A$ be skew-symmetric. 
Then $r$ is a solution of the AD-YBE in $(A,\succ,\prec)$ if and only if $r$
is an $\mathcal O$-operator of the associative algebra $(A, \cdot)$ associated to $(A^{*},-R_{\prec}^{*},-L_{\succ}^{*})$.
 \end{ex}
 
 \begin{defi} Let $(A,\succ,\prec)$ be an anti-dendriform algebra and $(V,l_{\succ},r_{\succ},l_{\prec},r_{\prec})$ be
 its representation.
  An $\mathcal O$-operator $T$ of $(A,\succ,\prec)$ associated to $(V,l_{\succ},r_{\succ},l_{\prec},r_{\prec})$
   is a linear map $T:V\longrightarrow A$ satisfying
\begin{equation*}T(u)\succ T(v)=T (l_{\succ}(T(u))v+r_{\succ}(T(v))u),\ \ \ 
T(u)\prec T(v)=T (l_{\prec}(T(u))v+r_{\prec}(T(v))u),~~\forall~u,v\in V.\end{equation*}
\end{defi}

\begin{thm}
 Let $(A,\succ,\prec)$ be an anti-dendriform algebra and $(V,l_{\succ},r_{\succ},l_{\prec},r_{\prec})$ be
  a representation of $(A,\succ,\prec)$. Let
$(V^{*},-r_{\cdot}^{*},l_{\prec}^{*},r_{\succ}^{*},-l_{\cdot}^{*})$ 
be the dual representation of $A$ given by Proposition \ref{zr}. Let $\hat{A}=A\ltimes V^{*}$ and
 $T:V\longrightarrow A$ be a linear map which is identifies an element in $\hat{A}\otimes \hat{A}$ through
 ($Hom(V,A)\simeq A\otimes V^{*}\subseteq \hat{A}\otimes \hat{A}$).
  Then $r=T-\tau(T)$ is a skew-symmetric solution
of the AD-YBE in the anti-dendriform algebra $\hat{A}$ if and only if $T$ is an $\mathcal O$-operator
of $(A,\succ,\prec)$ associated with  $(V,l_{\succ},r_{\succ},l_{\prec},r_{\prec})$, where
$r_{\prec}^{*}+r_{\succ}^{*}=r_{\cdot}^{*}$ and $l_{\cdot}^{*}=l_{\prec}^{*}+l_{\succ}^{*}$
\end{thm}
\begin{proof}
For all $x+a^{*},y+b^{*}\in \hat{A}$ with $x,y\in A$ and $a^{*},b^{*}\in V^{*}$, the anti-dendriform algebraic structure 
$(\succ,\prec)$ on $\hat{A}$ is defined as follows:
\begin{align}\label{YE9}
 &(x+a^{*})\succ(y+b^{*})=x\succ y-r_{\cdot}^{*}(x)b^{*}+l_{\prec}^{*}(y)a^{*},
 \\& \label{YE10}(x+a^{*})\prec(y+b^{*})=x\prec y+r_{\succ}^{*}(x)b^{*}-l_{\cdot}^{*}(y)a^{*},\end{align}
 and the associated associative algebraic structure $\cdot$ is given by
 \begin{align*}
 (x+a^{*})\cdot(y+b^{*})=x\cdot y-r_{\prec}^{*}(x)b^{*}-l_{\succ}^{*}(y)a^{*}.\end{align*}
 Assume that $\{v_1,v_2,\cdot\cdot\cdot, v_n\}$ is a basis of $V$ and 
$\{v_1^{*},v^{*}_2,\cdot\cdot\cdot, v^{*}_n\}$ is the dual basis of $V^{*}$. Then 
$T=\sum_{i=1}^{n}T(v_i)\otimes v^{*}_i\in T(V)\otimes V^{*}\subseteq \hat{A}\otimes \hat{A}$. Note that
 \begin{align}\label{YE11}
 &l_{\succ}^{*}(T(v_{i})v_{j}^{*})
 =\sum_{k=1}^{n}\langle v_{j}^{*},l_{\succ}(T(v_{i})v_k\rangle v_{k}^{*}, \ \ \ r_{\succ}^{*}(T(v_{i})v_{j}^{*})=
 \sum_{k=1}^{n}\langle v_{j}^{*},r_{\succ}(T(v_{i})v_k\rangle v_{k}^{*},\\&
 \label{YE12} l_{\prec}^{*}(T(v_{i})v_{j}^{*})=
 \sum_{k=1}^{n}\langle v_{j}^{*},l_{\prec}(T(v_{i})v_k\rangle v_{k}^{*}
, \ \ \ r_{\prec}^{*}(T(v_{i})v_{j}^{*})= \sum_{k=1}^{n}\langle v_{j}^{*},r_{\prec}(T(v_{i})v_k\rangle v_{k}^{*}.\end{align}
 Using Eqs.\eqref{YE9}-\eqref{YE12}, we have
 \begin{align*}
 r_{12}\cdot r_{13}
 &=\sum_{i,j=1}^{n}T(v_i)\cdot T(v_j)\otimes v_{i}^{*}\otimes v_{j}^{*}-
 T(v_i)\cdot v_{j}^{*}\otimes v_{i}^{*}\otimes T(v_{j})
-v_{i}^{*}\cdot T(v_j)\otimes T(v_{i})\otimes v_{j}^{*}
\\&=\sum_{i,j=1}^{n}T(v_i)\cdot T(v_j)\otimes v_{i}^{*}\otimes v_{j}^{*}
+r_{\prec}^{*}(T(v_i))v_{j}^{*}\otimes v_{i}^{*}\otimes T(v_{j})
+l_{\succ}^{*}(T(v_j))v_{i}^{*}\otimes T(v_{i})\otimes v_{j}^{*}
\\&=\sum_{i,j=1}^{n}T(v_i)\cdot T(v_j)\otimes v_{i}^{*}\otimes v_{j}^{*}
+v_{j}^{*}\otimes v_{i}^{*}\otimes T(r_{\prec}(T(v_i))v_{j})
+v_{i}^{*}\otimes T(l_{\succ}(T(v_j))v_{i})\otimes v_{j}^{*},\end{align*}
\begin{align*}
- r_{13}\prec r_{23}
 &=\sum_{i,j=1}^{n}T(v_i)\otimes v_{j}^{*}\otimes v_{i}^{*}\prec  T(v_j)-
v_{i}^{*}\otimes v_{j}^{*}\otimes (T(v_{i})\prec T(v_{j}))
+v_{i}^{*}\otimes T(v_j)\otimes (T(v_{i})\prec v_{j}^{*})
\\&=\sum_{i,j=1}^{n}-T(v_i)\otimes v_{j}^{*}\otimes l_{\cdot}^{*}(T(v_j))v_{i}^{*}-
v_{i}^{*}\otimes v_{j}^{*}\otimes (T(v_{i})\prec T(v_{j}))
+v_{i}^{*}\otimes T(v_j)\otimes r_{\succ}^{*}(T(v_{i}) v_{j}^{*})
\\&=\sum_{i,j=1}^{n}-T(l_{\cdot}(T(v_j))v_i)\otimes v_{j}^{*}\otimes v_{i}^{*}-
v_{i}^{*}\otimes v_{j}^{*}\otimes (T(v_{i})\prec T(v_{j}))
+v_{i}^{*}\otimes T(r_{\succ}(T(v_{i})v_j)\otimes  v_{j}^{*})
,\end{align*}
\begin{align*}
 r_{23}\prec r_{12}
 &=\sum_{i,j=1}^{n}T(v_j)\otimes (T(v_i)\succ v_{j}^{*})\otimes v_{i}^{*}+
v_{j}^{*}\otimes (v_{i}^{*}\succ T(v_j))\otimes T(v_{i})
-v_{j}^{*}\otimes (T(v_i)\succ T(v_j)))\otimes v_{i}^{*})
\\&=\sum_{i,j=1}^{n}-T(v_j)\otimes r_{\cdot}^{*} (T(v_i))v_{j}^{*})\otimes v_{i}^{*}+
v_{j}^{*}\otimes l_{\prec}^{*}(T(v_j))v_{i}^{*}\otimes T(v_{i})
-v_{j}^{*}\otimes (T(v_i)\succ T(v_j))\otimes v_{i}^{*})
\\&=\sum_{i,j=1}^{n}-T(r_{\cdot} (T(v_i))v_j)\otimes v_{j}^{*})\otimes v_{i}^{*}+
v_{j}^{*}\otimes v_{i}^{*}\otimes T(l_{\prec}(T(v_j))v_{i})
-v_{j}^{*}\otimes (T(v_i)\succ T(v_j))\otimes v_{i}^{*})
.\end{align*}
Therefore, the conclusion holds.
\end{proof}

%%%%%%%%%%%%%%%%%%%%%%%%%%%%%%%%%%%%%%%%%%%%%%%%%%%%%%%%%%%%%%%%%%%%%%%%%%%%%%%%%%%%%%%%%%%%%%%%%%%%%%%%%%

\begin{center}{\textbf{Acknowledgments}}
\end{center}
This work was supported by the Natural Science
Foundation of Zhejiang Province of China (LY19A010001), the Science
and Technology Planning Project of Zhejiang Province
(2022C01118).

%%%%%%%%%%%%%%%%%%%%%%%%%%%%%%%%%%%%%%%%%%%%%%%%%%%%%%%%%%%%%%%%%%%%%%%%%%%%%%%%%%%%%%%%%%%%%%%%%%%%%%%%%%
\begin{center} {\textbf{Statements and Declarations}}
\end{center}
 All datasets underlying the conclusions of the paper are available
to readers. No conflict of interest exits in the submission of this
manuscript.

%%%%%%%%%%%%%%%%%%%%%%%%%%%%%%%%%%%%%%%%%%%%%%%%%%%%%%%%%%%%%%%%%%%%%%%%%%%%%%%%%%%%%%%%%%%%%%%%%%%%%%%%%%

\end {document}